\documentclass[openright, 12pt]{article}
\usepackage{amsmath,amsfonts}
\usepackage{amsthm}
\newtheorem{theorem}{Theorem}
\newtheorem{proposition}[theorem]{Proposition}
\newtheorem{corollary}[theorem]{Corollary}
\newtheorem{lemma}[theorem]{Lemma}
\newtheorem{remark}[theorem]{Remark}
\newtheorem{example}[theorem]{Example}
\newtheorem{definition}[theorem]{Definition}
\newtheorem{con}[theorem]{Conjecture}

\usepackage{float} 
\usepackage{multirow}
\usepackage{graphicx}
\usepackage{psfrag}
\usepackage{enumerate}
\usepackage{enumitem}
\usepackage{mathtools}
\numberwithin{equation}{section}
\numberwithin{theorem}{section}
\usepackage{enumitem}

\usepackage{hyperref, mathrsfs}
\usepackage{xcolor}
\hypersetup{colorlinks}
\hypersetup{citecolor=blue}
\hypersetup{urlcolor=blue}

\DeclareMathOperator{\diam}{diam}
\DeclareMathOperator{\id}{Id}
\DeclareMathOperator{\iso}{Iso}
\DeclareMathOperator{\kker}{Ker}
\DeclareMathOperator{\leng}{length}
\DeclareMathOperator{\genlin}{GL}
\DeclareMathOperator{\Adjoint}{Ad}
\DeclareMathOperator{\Tan}{T}
\DeclareMathOperator{\thi}{thick}
\DeclareMathOperator{\dimensi}{dim}

\usepackage{amssymb}
\usepackage[latin1]{inputenc}

\usepackage{tikz-cd}
\usepackage{pgf}
\usepackage{mathtools}

\begin{document}


\begin{center}
\large{Limits of almost homogeneous spaces and 

their fundamental groups}

\normalsize

Sergio Zamora

zamora@mpim-bonn.mpg.de
\end{center}

\begin{abstract}
We say that a sequence of proper geodesic spaces $X_n$ consists of \textit{almost homogeneous spaces} if there is a sequence of discrete groups of isometries $G_n \leq \iso(X_n)$ with $\diam (X_n/G_n)\to 0$ as $n \to \infty$.

We show that if a sequence $(X_n,p_n)$ of pointed almost homogeneous spaces converges in the pointed Gromov--Hausdorff sense to a space $(X,p)$, then $X$ is a nilpotent locally compact group equipped with an invariant geodesic metric. 

Under the above hypotheses, we show that if $X$ is semi-locally-simply-connected, then it is a nilpotent Lie group equipped with an invariant sub-Finsler metric, and for $n$ large enough,  $\pi_1(X) $ is a subgroup of a quotient of $ \pi_1(X_n) $.

\end{abstract}

\section{Introduction}\label{sec:intro}

We say that a sequence of proper geodesic spaces $X_n$ consists of \textit{almost homogeneous spaces} if there is a sequence of discrete groups of isometries $G_n \leq \iso(X_n)$ with $\diam (X_n/G_n)\to 0$ as $n \to \infty$.

\begin{remark}
\rm A sequence of homogeneous spaces $X_n$ does not necessarily consist of almost homogeneous spaces, since the groups $ \iso (X_n ) $ are not necessarily discrete.
\end{remark}

\begin{example}
 \rm Let $Z_n$ be a sequence of compact geodesic spaces with diam$(Z_n)\to 0$ as $n \to \infty$. If $\tilde{Z}_n \to Z_n$ is a sequence of regular covers, then the sequence $\tilde{Z}_n$ consists of almost homogeneous spaces. 
\end{example}

The goal of this paper is to understand the Gromov--Hausdorff limits of sequences of almost homogeneous spaces. In the case when the sequence consists of blow-downs of a single space, the problem was solved by Mikhail Gromov and Pierre Pansu \cite{GromovPG, Pansu}.

\begin{theorem}
\rm (Gromov--Pansu) Let $(X,p)$ be a pointed proper geodesic space, and $G \leq \iso (X)$ a discrete group of isometries with $\diam (X/G) < \infty$. If for some sequence of positive numbers $\lambda_n \to \infty$, one has
\[  \lim_{n \to \infty} \left( \dfrac{1}{\lambda_n}X, p \right)  =  (Y,q)   \]
in the pointed Gromov--Hausdorff sense, then $Y$ is a simply connected nilpotent Lie group equipped with a Carnot--Caratheodory metric (a Carnot--Caratheodory metric is a special kind of invariant sub-Finsler metric  satisfying that for any $\lambda > 0$, the space $\lambda Y$ is isometric to $Y$).
\end{theorem}

When the limit is compact, Alan Turing solved the finite dimensional case \cite{Tur}, and using Turing's result, Tsachik Gelander solved the infinite dimensional case \cite{Gel}.

\begin{theorem}
\rm (Turing--Gelander) Let $X_n$ be a sequence of almost homogeneous spaces. If the sequence $X_n$ converges in the Gromov--Hausdorff sense to a compact space $X$, then $X$ is a (possibly infinite dimensional) torus equipped with an invariant metric. 
\end{theorem}

The main result of this paper deals with the case in which the limit is non-compact.

\begin{theorem}\label{thm:main-theorem}
\rm Let $(X_n, p_n)$ be a sequence of pointed almost homogeneous spaces. If the sequence $(X_n, p_n)$ converges in the pointed Gromov--Hausdorff sense to a space $(X,p)$, then $X$ is a nilpotent group equipped with an invariant metric. Furthermore, if $X$ is semi-locally-simply-connected, then it is a Lie group equipped with a sub-Finsler metric, and  for $n$ large enough, $\pi_1(X)$ is a subgroup of a quotient of $\pi_1(X_n)$.
\end{theorem}

\begin{remark}
 \rm  The hypothesis of $X$ being semi-locally-simply-connected can be replaced by $X$ having finite topological dimension, because of the following result (solution to Hilbert's fifth problem) by Deane Montgomery and Leo Zippin \cite{MZ}.
\end{remark}

\begin{theorem}\label{thm:mz-hilbert}
\rm  (Montgomery--Zippin) Let $X$ be a homogeneous proper geodesic space. If $X$ has finite topological dimension, then it is homeomorphic to a topological manifold, and its isometry group is a Lie group.
\end{theorem}

\subsection{Lower semi-continuity of $\pi_1$}

In the context of Theorem \ref{thm:main-theorem}, if $X$ is semi-locally-simply-connected then $\pi_1(X)$ is in some sense not larger than $\pi_1(X_n)$.  This is an instance of a more general phenomenon (see \cite[Section 3E]{GromovMS}).

\begin{theorem}\label{thm:lsfg-classic}
\rm  (Folklore) Let $X_n$ be a sequence of compact geodesic spaces. If the sequence $X_n$ converges to a compact semi-locally-simply-connected space $X$,  then for $n$ large enough, $\pi_1(X)$ is a quotient of $\pi_1(X_n)$.
\end{theorem}  
 
 This property is further studied by Christina Sormani and Guofang Wei in  \cite{SWHau, SWUni, SWCov}. This result fails if the limit is not compact  \cite[Example 1.2]{SWUni}, or not semi-locally-simply-connected \cite[Example 2.6]{SWHau}. The following example shows that if one works with homogeneous spaces instead of almost homogeneous spaces in Theorem \ref{thm:main-theorem}, it may happen that $\pi_1(X_n)$ is trivial for all $n$ and $X$ is semi-locally-simply-connected, but $\pi_1(X)$ is non-trivial (see \cite[Examples 3.11.$a_+$]{GromovMS}).

\begin{example}
\rm  Let $Y$ be $\mathbb{S}^1$ with its standard metric of length $2 \pi$  and $Z_n $ be $\mathbb{S}^3$ with the round (bi-invariant) metric of constant curvature $1/n^2$.  Let $X_n $ be the quotient $(Y \times Z_n ) / \mathbb{S}^1$ where $\mathbb{S}^1$ acts on $Y \times Z_n$ as follows:
\[  z(w,q) = (wz ^{-1}, z q): z, w \in \mathbb{S}^1, q \in \mathbb{S}^3. \]
Then $X_n$ is isometric to $\mathbb{S}^3$ equipped with a re-scaled Berger metric. The sequence $X_n$ consists of simply connected homogeneous spaces, but its pointed Gromov--Hausdorff limit is $\mathbb{S}^1 \times \mathbb{R}^2$, which is not simply connected.
\end{example}

Under the hypotheses of Theorem \ref{thm:main-theorem}, assuming $X$ is semi-locally-simply-connected, one may wonder whether $\pi_1(X)$ is a quotient of $\pi_1(X_n)$ for $n$ large enough (without passing to a subgroup) just like in Theorem \ref{thm:lsfg-classic}. The following example shows that it is not the case.

\begin{example}
\rm  Define a ``dot product'' $\mathbb{R}^4 \times \mathbb{R}^4 \to \mathbb{R}^6$ in $\mathbb{R}^4$ as
\[  (a \cdot b) : = (a_1b_2, a_1b_3, a_1b_4, a_2b_3, a_2b_4,  a_3b_4),    \text{ } a,b \in \mathbb{R}^4  .     \]
With it, define a group structure on $\mathbb{R}^4 \times \mathbb{R}^6$ as
\[  (a,x) \cdot (b,y) : = (a+b , x+y+ (a\cdot b)), \text{ }a,b \in \mathbb{R}^4 , \text{ } x,y \in \mathbb{R}^6 . \]
Let $G$ be the above group equipped with a left invariant Riemannian metric. For each $n$, define the subgroups $K \leq K_n \triangleleft G_n \leq G \cong \mathbb{R}^4 \times \mathbb{R}^6 $ as:
\begin{align*}
 K & : =  \{ 0 \}  \times \mathbb{Z}^6     \\
  K_n & : =   ( n \mathbb{Z}^4 ) \times \mathbb{Z}^6   \\
    G_n  & : =   \left( \frac{1}{n} \mathbb{Z}^4  \right) \times \left( \frac{1}{n^2} \mathbb{Z}^6  \right)   .
\end{align*} 
The sequence $X_n : = G / K_n$ converges in the pointed Gromov--Hausdorff sense to $X:= G/K$. Since the sequence of finite groups $G_n / K_n$ acts on $X_n$  with $\diam (X_n/(G_n/K_n))\to 0$, the sequence $X_n$ consists of almost homogeneous spaces. A direct computation shows that the abelianization of $\pi_1(X_n ) = K_n$ is isomorphic to 
\[ \mathbb{Z}^4 \oplus  \left( \mathbb{Z} / n^2 \mathbb{Z}  \right)^6 . \]
 Then it is easy to see that $\pi_1(X) = K \cong \mathbb{Z}^6$ is not a quotient of $\pi_1(X_n)$ for any $n$.
\end{example}

\subsection{Existence of the limit}

Given  a sequence $X_n$ of almost homogeneous spaces, one may wonder which conditions guarantee the existence of a convergent subsequence. For example, if the spaces $X_n$ are Riemannian manifolds with Ricci curvature $\geq K$ and dimension $\leq N$ for some $K \in \mathbb{R}$, $N \in \mathbb{N}$, Gromov compactness criterion implies the existence of a convergent subsequence \cite[Theorem 5.3]{GromovMS}. 

 Itai Benjamini, Hilary Finucane and Romain Tessera found another sufficient condition for this partial limit to exist when the spaces $X_n$ are graphs \cite{BFT}.

\begin{theorem}\label{thm:bft}
\rm  (Benjamini--Finucane--Tessera) Let $D_n \leq \Delta _n$ be two sequences going to infinity, and let $(X_n,p_n)$ be a sequence of pointed graphs. Assume there is a sequence of discrete groups $G_n \leq \iso (X_n)$ acting transitively on the sets of vertices. If the balls of radius $D_n$ in $X_n$ satisfy
\[ \vert B^{X_n}(p_n,D_n)\vert =O(D_n^q) \]
for some $q>0$, then the sequence of pointed almost homogeneous spaces
\[ \left( \dfrac{1}{\Delta_n} X_n ,p_n \right)  \] 
 has a subsequence converging in the pointed Gromov--Hausdorff sense to a nilpotent Lie group.
\end{theorem}

\begin{remark}
\rm  Recently, Romain Tessera and Matthew Tointon showed that the hypothesis in Theorem \ref{thm:bft} of the groups $G_n$ being discrete can be removed \cite{TT}. Moreover, one  could define a sequence of \textit{weakly almost homogeneous spaces} to be a sequence of proper geodesic spaces $X_n$ with groups of isometries $G_n \leq \iso (X_n)$ acting with discrete orbits and such that diam$(X_n / G_n ) \to 0$.  Their results imply that Theorem \ref{thm:main-theorem} holds under the weaker assumption that the spaces $X_n$ are weakly almost homogeneous.
\end{remark}

\subsection{Further problems}

There are two natural strengthenings of Theorem \ref{thm:main-theorem}. One could ask if a weaker conclusion holds if one removes the hypotheses of the groups $G_n$ acting almost transitively, or the limit $X$ being semi-locally-simply-connected.

\begin{con}
 \rm Let $(X_n,p_n)$ be a sequence of pointed simply connected proper geodesic spaces. Assume there is a sequence of discrete groups of isometries $G_n \leq \iso (X_n)$ with $\diam (X_n/G_n) \leq C$ for some $C>0$. If $(X_n,p_n)$ converges in the pointed Gromov--Hausdorff sense to a pointed semi-locally-simply-connected space $(X,p)$ for which $\iso (X)$ is a Lie group, then $X$ is simply connected.  
\end{con}

\begin{con}
 \rm Let $(X_n,p_n)$ be a sequence of pointed simply connected almost homogeneous spaces. If $(X_n,p_n)$ converges in the pointed Gromov--Hausdorff sense to a pointed space $(X,p)$, then $X$ is simply connected.
\end{con}

\subsection{Summary}

The proof of Theorem \ref{thm:main-theorem} consists of three parts. In the first part, we show that $X$ is a nilpotent group equipped with an invariant metric (Theorem \ref{thm:part-i}). In the second part, we show that if $X$ is semi-locally-simply-connected, then it is a Lie group (Theorem \ref{thm:part-ii}). In the third part, we show that if $X$ is a Lie group, then for large enough $n$, there are quotients of $ \pi_1(X_n)$ containing isomorphic copies of $ \pi_1 (X)$ (Theorem \ref{thm:part-iii}).

In Section \ref{sec:prelim}, we present the relevant definitions and preliminary results required for the proof of Theorem \ref{thm:main-theorem}. In Section \ref{sec:part-i}, by repeated applications of a Margulis Lemma by  Breuillard--Green--Tao \cite{BGT} we find almost nilpotent discrete groups of isometries $G_n^{\prime} \leq \iso (X_n)$ acting almost transitively (Lemma \ref{lem:almost-nilpotent}). Combining this with the Gleason--Yamabe structure theory of locally compact groups \cite{Gleason, Yamabe}, and a result of Berestovskii about groups of isometries of homogeneous spaces \cite{BerBus}, we obtain Theorem \ref{thm:part-i}.  In Section \ref{sec:slsc} we prove  Theorem \ref{thm:part-ii} using elementary Lie theory and algebraic topology techniques.  

Sections \ref{sec:trans} to \ref{sec:dictionary} contain the proof of Theorem \ref{thm:part-iii}, finishing the proof of Theorem \ref{thm:main-theorem}. In Section \ref{sec:trans} we use commutator estimates similar to the ones in \cite{BK, GrAF} to prove that the groups $G_n^{\prime}$ act ``almost by translations'' on the spaces $X_n$ (Proposition \ref{prop:connected}). In Section \ref{sec:nss}, we use the escape norm from  \cite{BGT} to find small normal subgroups $W_n \triangleleft G_n^{\prime}$  with the property that the spaces $X_n$ and $X_n/ W_n$ are globally Gromov--Hausdorff close, and the groups $\Gamma_n := G_n^{\prime}/W_n$ contain large subsets $A_n$ without non-trivial subgroups. 

In Section \ref{sec:nilprog}, we use the space $X$ as a model (as defined by Hrushovski \cite{Hr}) for the ultralimit 
\[ \textbf{A}  := \lim\limits_{n \to \alpha} A_n.\]
 This enables us to find large nice subsets $P_n$ of $\Gamma _n$ (nilprogressions in $C$-normal form, as defined by Breuillard--Green--Tao \cite{BGT}).

In Section \ref{sec:malcev}, we use the Malcev Embedding Theorem \cite{Malcev} to find  groups $ \tilde{\Gamma}_{n}$ isomorphic to lattices in simply connected Lie groups, with isometric actions
\[  \Phi_n: \tilde{\Gamma}_{n}  \to  \iso ( X_n /W_n). \]
Using elementary algebraic topology, we show that the kernels $\kker (\Phi_n)$ of those actions are isomorphic to quotients of $\pi_1(X_n).$  Finally, in section \ref{sec:dictionary}, we find subgroups of $\kker (\Phi_n)$ isomorphic to $\pi_1 (X)$,  finishing the proof of Theorem \ref{thm:part-iii}, and consequently Theorem \ref{thm:main-theorem}.

\subsection{Acknowledgments}

The author would like to thank Vladimir Finkelshtein, Enrico LeDonne, Adriana Ortiz, Anton Petrunin, and Burkhard Wilking for helpful comments and discussions. The author would also like to thank an anonymous reviewer whose comments have improved this paper significantly. This research was supported in part by the International Centre for Theoretical Sciences (ICTS) during a visit for participating in the program - Probabilistic Methods in Negative Curvature (Code: ICTS/pmnc 2019/03).

\section{Preliminaries}\label{sec:prelim}

\subsection{Notation}

 For $H,K$ subgroups of a group $G$, we define their \textit{commutator subgroup} $[H,K]$ to be the group generated by the elements $[h,k]:=h^{-1}k^{-1}hk$ with $h \in H, k \in K$. Define $G^{(0)} $ as $G$, and $G^{(j+1)}$ inductively as $G^{(j+1)} := [G, G^{(j)}] $. If $G^{(s)} = \{ e\}$ for some $s \in \mathbb{N}$, we say that $G$ is \textit{nilpotent of step} $\leq s$. 
 
 We say that a set $A \subset G$ is \textit{symmetric} if $A=A^{-1} $ and $e \in A$. For subsets $A_1, \ldots, A_k \subset G$, we denote by $A_1 \cdots A_k $ the set of all products 
\[ \{ a_1 \cdots  a_k \vert a_i \in A_i \} \subset G, \] 
 and  by $A_1 \times \ldots \times A_k $ the set of all sequences
\[ \{ (a_1, \ldots,  a_k) \vert a_i \in A_i \} \subset G^k.    \]
If $A_i = A $ for $i \in \{ 1, \ldots , k\}$, we will also denote $  A_1 \cdots A_k  $ by $A^k$, and $  A_1 \times \ldots \times A_k  $ by $A^{  \times k  }$.\\

Let $X$ be a topological space and $\beta, \gamma : [0,1] \to X$ two curves. We denote by $\overline{\beta} : [0,1] \to X$ the curve given by $\overline{\beta} (t) =  \beta (1-t) $. And if $\beta(1)= \gamma (0)$, we denote by $\beta \ast \gamma : [0,1] \to X$ the concatenation
\begin{center}
$    \beta \ast \gamma (t) = \begin{cases} 
      \beta (2t) & \text{ if } t \leq 1/2 \\
      \gamma (2t-1) & \text{ if }t \geq 1/2. 
   \end{cases}   $
\end{center}
If $\beta(1)\neq \gamma (0)$, we say that $\beta \ast \gamma$ is undefined. \\

In a metric space $(X,d)$, we will denote the open ball of center $p\in X$ and radius $r > 0$ as  $ B_{d}^X(p,r) $. We will sometimes omit $d$ or $X$ and write $B(p,r)$ if the metric space we are considering is  clear from the context.

\subsection{Groups of isometries}

For a pointed proper metric space $(X,p)$, we equip its isometry group $\iso (X)$ with the metric $d_0^{p}$ given by
\begin{equation}\label{eq:d0}
d_{0}^{p} (h_1, h_2) : = \inf_{r > 0 } \left\{ \frac{1}{r} + \sup_{x \in B(p,r)} d(h_1x, h_2x)  \right\}          
\end{equation}
for $h_1, h_2 \in \iso (X)$. It is easy to see that this metric is left invariant, induces the compact-open topology, and makes $\iso (X)$ a proper metric space.

 \begin{definition}
\rm   Let $f : X \to Y$ be a map between proper geodesic spaces. We say that $f$ is a \textit{metric submersion} if for every $x \in X $ and $r > 0$, the image of the closed ball of radius $r$ around $x$ is the closed ball of radius $r$ around $f(x)$.
 \end{definition}

\begin{lemma}\label{lem:submersion}
\rm  Let $X$ be a proper geodesic space, and $G \leq \iso (X)$  a closed subgroup. Then the quotient map $f: X \to X/ G$ is a metric submersion.
\end{lemma}

\begin{proof}
Let $x \in X $, $r > 0 $, and  $y \in X/G$ with $d_{X/G}( f(x), y) \leq  r$. Since $G$ is closed, the orbits are closed, and since $X$ is proper, there is $z \in f^{-1}(y)$ with $d(x,z) = d(f(x), y) \leq  r$. This proves that the image of the closed ball of radius $r$ around $x$ contains the closed ball fo radius $r$ around $f(x)$. The other contention is clear as $f$ is $1$-Lipschitz.
\end{proof}

\begin{lemma}\label{lem:horizontal-lift}
\rm  Let $f: X \to Y$ be a metric submersion between proper geodesic spaces, $\gamma : [0,1] \to Y$ a Lipschitz curve, and $p \in f^{-1}(\gamma (0))$. Then there is a curve $\tilde{\gamma} : [0,1] \to X$ with $f \circ \tilde{ \gamma} = \gamma ,$  $\tilde{\gamma}(0) = p$,  and $\leng (\tilde{\gamma}) = \leng (\gamma )$. 
\end{lemma}

\begin{proof}
For each $j \in \mathbb{N}$, let $D_j : = \{ 0, \frac{1}{2^j}, \ldots , \frac{2^j-1}{2^j} , 1 \} $, and define $h_j : D_j \to X$ as follows:  Let $h_j(0)=p$, and inductively, let $h_j(x + 1/2^j)$ be a point in $f^{-1}( \gamma (x + 1/2^j ) )$, such that
\[   d_{X} (  h_j ( x + 1/2^j), h_j ( x )     )  = d_{Y} (  \gamma ( x + 1/2^j  ), \gamma (x)   ), \text{ for }x \in D_j \backslash \{ 1 \} .    \]
Using Cantor's diagonal argument, we can find a subsequence of $h_j$ that converges for every dyadic rational. Since the maps $h_j$ are uniformly Lipschitz, we can extend this limit map to a Lipschitz map $\tilde{ \gamma} : [0,1] \to X$. It is easy to check that $\tilde{\gamma}$ satisfies the desired properties.
\end{proof}

\begin{lemma}\label{lem:short-powers}
\rm Let $(X,p)$ be a pointed metric space, $r > 0 $, and  
\[   S : = \{   g \in \iso (X)  \vert d(gp,p ) \leq  r \}. \]
Then for each $m \in \mathbb{N}$ one has
\[   S^m \subset \{ g \in \iso (X) \vert d(gp,p) \leq m r \}.              \]
\end{lemma}
\begin{proof}
For $s_1, \ldots , s_m \in S$, by the triangle inequality we have
\[ d(s_1 \cdots s_m p, p)  \leq \sum_{j=1}^m d(s_1 \cdots s_j p, s_1\cdots s_{j-1}  p)  = \sum_{j=1}^m d(s_j p , p )  \leq mr.\]  
\end{proof}

\begin{lemma}\label{lem:short-generators}
\rm Let $(X,p)$ be a pointed proper geodesic space, and $G \leq \iso (X)$ a closed subgroup. Then the set 
\[   S : = \{   g \in G \vert d(gp,p ) \leq 3\cdot \diam (X/G) \} \]
generates $G$. 
\end{lemma}

\begin{proof}
Let $g \in G$. Since $X$ is geodesic, there is a sequence of points $p=p_0, p_1, \ldots , p_k = gp \in X$ satisfying that $d(p_j,p_{j-1}) \leq \diam (X/G)$ for each $j$. Choose a sequence $g_1, \ldots, g_{k-1} \in G$ such that $d(g_jp,p_j)\leq \diam (X/G)$ for each $j$, and let $g_0 : = \id _X$, $g_k : = g$. From the triangle inequality, for each $j \in \{ 1, \ldots , k\}$ we have $g_{j-1}^{-1}g_j \in S$. Then 
\[  g = g_1  ( g_1^{-1}g_2) \cdots  (g_{k-1}^{-1}g_k) \in S^k . \]
\end{proof}
\begin{lemma}\label{lem:diameter-subgroup}
\rm Let $X$ be a proper geodesic space, and $H \leq G \leq \iso (X)$ be closed subgroups. Then 
\[  \text{diam}(X/H) \leq 3 [G: H ]\cdot \diam (X/G) .   \]
\end{lemma}
\begin{proof}
Let $x \in X$, and define 
\[  S : = \{   g \in G \vert d(gx,x) \leq 3 \cdot \diam (X/G)  \}, \]
which generates $G$ by Lemma \ref{lem:short-generators}. If  we have $S^{k+1} H = S^k H$ for some $k \in \mathbb{N}$, then  an inductive argument implies that
\[S^kH = \bigcup_{n \in \mathbb{N}} S^n H = G .\]
Since there are only $[G:H]$ cosets, then $S^{k+1}H = S^k H$  for $k \geq [G:H]-1$. This implies that $S^{[G:H]-1}$ intersects each $H$-coset. 

For any $y \in X$, there is $g \in G$ with $d(x,gy)\leq \diam (X/G)$, and by the above analysis, there is $u \in S^{[G:H]-1}$ with $u^{-1}g \in H$. Also,
\begin{align*}
d(x,u^{-1}gy) & =  d(ux,gy)\\
& \leq  d(ux,x) + d(x,gy)\\
& \leq   3  ([G:H]-1) \cdot \diam (X/G) + \diam (X/G),
\end{align*} 
where we used Lemma \ref{lem:short-powers} in the third line.  Since $x,y \in X$ were arbitrary, the result follows.
\end{proof}
\begin{lemma}\label{lem:normal-stabilizers}
\rm Let $(X,p)$ be a pointed proper geodesic space,  $G \leq \iso (X)$ a closed group, and $H \triangleleft G$ a normal subgroup. If for some $\varepsilon > 0 $, one has $d(hp,p) \leq \varepsilon$ for all $h \in H$, then for all $h \in H$, $x \in X$, one has 
\[  d(hx,x) \leq \varepsilon + 2 \cdot \diam (X/G) .              \]
 In particular, if $H$ is contained in the stabilizer of $p$, and $G$ acts transitively on $X$, then $H$ is trivial.
\end{lemma}
\begin{proof}
For $ x \in X$, there is $g \in G$ with $d(gx,p) \leq \diam (X/ G)$. Then for $h \in H$ one has by the triangle inequality,
\begin{align*}
d(hx,x) & = d(ghx,gx) \\
& \leq d(ghg^{-1}gx, ghg^{-1} p) + d(ghg^{-1}p,p) + d(p, gx)\\
& \leq \diam (X/ G) + \varepsilon + \diam (X/G),
\end{align*}
where we used that $H$ is normal in the third line. Since $x \in X$, $h \in H$ were arbitrary, the lemma follows.
\end{proof}

The following result by Valerii Berestovskii concerns groups that act transitively on geodesic spaces \cite[Theorem 1]{BerBus}.

\begin{theorem}\label{thm:berestovskii-open}
\rm (Berestovskii) Let $X$ be a proper geodesic space, and $G$ a closed subgroup of $\iso (X)$ acting transitively. If $\mathcal{O} \leq G$ is an open subgroup, then $\mathcal{O}$ also acts transitively on $X$. 
\end{theorem}

\subsection{Lie groups and Hilbert's fifth problem}

\begin{definition}
\rm Let $\mathfrak{g}$ be a nilpotent Lie algebra. We say that an ordered basis $  \{ v_1, \ldots , v_r \} \hookrightarrow \mathfrak{g}$ is a \textit{strong Malcev basis} if for all $k \in \{1, \ldots ,r \} $, the vector subspace $J_k \leq \mathfrak{g}$ generated by $\{ v_{k+1}, \ldots , v_r \}$ is an ideal, and $v_k + J_k$ is in the center of $\mathfrak{g}/J_k$.
\end{definition}

\begin{theorem}\label{thm:simply-connected-nilpotent}
\rm Let $G$ be a simply connected nilpotent Lie group, $\mathfrak{g} $ its Lie algebra, and $ \mathfrak{V} = \{ v_1, \ldots , v_r \} \hookrightarrow \mathfrak{g}$ a strong Malcev basis. Then
\begin{itemize}
\item $\exp : \mathfrak{g} \to G$ is a diffeomorphism.
\item The map $ \psi : \mathbb{R}^r \to G$ given by $ \psi (x_1, \ldots , x_r) := \exp(x_1v_1) \cdots \exp (x_r v_r)      $ is a diffeomorphism.
\item After identifying $\mathfrak{g} $ with $\mathbb{R}^r$ via $\mathfrak{V}$, the maps $ ( \log \circ $ $ \psi )  : \mathfrak{ g} \to \mathfrak{g}$ and $ (  \psi ^{-1} \circ \exp ) : \mathfrak{ g} \to \mathfrak{g}$ are polynomial of degree bounded by a number depending only on $r$.
\end{itemize} 
\end{theorem}
The proof of Theorem \ref{thm:simply-connected-nilpotent} can be found in  \cite[Proposition 1.2.1 and Proposition 1.2.7]{Cor}. Although the bound on the degree of the polynomials in the third item is not explicitly stated there, it  follows from their proof.

\begin{corollary}\label{cor:aspherical}
\rm Let $G$ be a connected nilpotent Lie group. Then $\pi_1(G)$ is finitely generated torsion free abelian and $\pi_k(G) = 0$ for all $k \geq 2$. 
\end{corollary}

\begin{proof}
By Theorem \ref{thm:simply-connected-nilpotent}, the universal cover $\tilde{G} \to G$ is contractible so $\pi_1(G)$ is torsion free and $\pi_k(G) =  0$ for all $k \geq 2$ (see \cite{Luck} for example). Since fundamental groups of Lie groups are finitely generated and abelian, the result follows.
\end{proof}

\begin{lemma}\label{lem:compact-nilpotent}
\rm Let $G$ be a connected nilpotent Lie group, and $K \leq G$ a compact subgroup. Then $K$ is central in $G$.
\end{lemma}

\begin{proof}
Let $\mathfrak{g}$ be the Lie algebra of $G$. By Weyl's unitary trick, we can equip $\mathfrak{g}$ with an inner product for which the adjoint action $\Adjoint : K \to \genlin (\mathfrak{g})$ consists of orthogonal transformations. 

For $h \in K$, it follows from Engel's Theorem that $\Adjoint _h : \mathfrak{g} \to \mathfrak{g}$ is idempotent (all its eigenvalues are equal to 1). Since $\id _{\mathfrak{g}}$ is the only orthogonal idempotent transformation, $\Adjoint _h = \id_{\mathfrak{g}}$ and $h$ commutes with all elements of $G$. Since $h $ was arbitrary, the lemma follows.
\end{proof}

\begin{corollary}\label{cor:compact-nilpotent}
\rm A connected compact nilpotent Lie group is abelian.
\end{corollary}

\begin{definition}
\rm  We say that a continuous map $X \to Y$ between path connected topological spaces \textit{has no content} or \textit{has trivial content} if the induced map $\pi_1(X) \to \pi_1 (Y) $ is trivial. Otherwise, we say that the map \textit{has non-trivial content}.
\end{definition}

\begin{lemma}\label{lem:compact-content}
\rm Let $G$ be a connected nilpotent Lie group and $K$ a non-trivial connected compact subgroup. Then the inclusion $K \to G$ has non-trivial content.
\end{lemma}

\begin{proof}
By Corollary \ref{cor:compact-nilpotent}, $K$ is a torus, so $\pi_1(K)$ is non-trivial. By Lemma \ref{lem:compact-nilpotent}, $K$ is central in $G$, so $G/K$ is a nilpotent Lie group and $\pi_2 (G/K) = 0 $ by Corollary \ref{cor:aspherical}. From the homotopy groups long exact sequence of the fibration $ G \to G/K$, we extract the portion
\[  \pi_2 (G/K) \to \pi_1 (K) \to \pi_1 (G) ,  \]
from which the lemma follows.
\end{proof}

\begin{lemma}\label{lem:lift-neighborhood}
 \rm Let $G$, $\tilde{G}$ be connected Lie groups such that $\tilde{G}$ is a discrete extension of $G$ (i.e. there is a surjective continuous morphism $f: \tilde{G} \to G$ with discrete kernel). Assume $G$ and $\tilde{G}$ are equipped with invariant geodesic metrics for which $f$ is a local isometry. Let $\delta > 0$ be such that $B^{G}(e,\delta)$ contains no non-trivial subgroups, and the inclusion $B^{G}(e, \delta) \to G$ has no content. Then $B^{\tilde{G}}(e,\delta)$ contains no non-trivial subgroups and the inclusion $B^{\tilde{G}}(e, \delta) \to \tilde{G}$ has no content.
\end{lemma}
\begin{proof}
If there is a group $H \subset B^{\tilde{G}}(p, \delta)$, then its image $f(H)$ is a subgroup of $G$ contained in $B^{G}(e, \delta)$, so $f(H) = \{e \} \leq G$. If there is a non-trivial element $h \in H \backslash \{e \}$, we can take a shortest path $\gamma : [0,1] \to \tilde{G}$ from $e$ to $h$. The projection $f \circ \gamma$ would be a non-contractible loop in $G$ contained in $B^{G}(e, \delta )$, contradicting the hypothesis that $B^{G}(e, \delta) \to G$ has no content. Also, we can consider the commutative diagram

\[
  \begin{tikzcd}[column sep=3em]
   B^{\tilde{G}}(e, \delta) \arrow{r}{  } \arrow{d}{} & \tilde{G}  \arrow{d}{f} \\
  B^{G}(e,\delta) \arrow{r}{ } & G  
  \end{tikzcd}
  \]
Since $ f$ is a covering map, the right vertical arrow induces an injective map at the level of fundamental groups, and the bottom horizontal arrow has no content by hypothesis. Therefore the top horizontal arrow has trivial content as well.
\end{proof}

\begin{lemma}\label{lem:li-com-free}
\rm Let $G$ be a Lie group, $\mathfrak{g}$ its Lie algebra, and assume $\exp : \mathfrak{g} \to G$ is a diffeomorphism. Let $g_1, \ldots , g_{\ell} \in G$ be such that  $[g_i,g_j] = e$ for all $i,j \in \{ 1, \ldots , \ell \}$, and the set $\{ \log (  g_1 ), \ldots , \log ( g_{\ell } ) \} \subset \mathfrak{g}$ is linearly independent. Then the group 
\[ \langle   g_1 , \ldots , g_{\ell }  \rangle \leq G \]
is isomorphic to $\mathbb{Z}^{\ell }$.
\end{lemma}

\begin{proof}
If $G$ is abelian, the result is trivial, so the goal is to reduce the general case to the abelian one. For each $i \in \{ 1, \ldots , \ell \}$, set $v_i : = \log (g_i)$, and let $\mathfrak{a}$ be the linear span of $   \{ v_1 , \ldots , v_{\ell } \} $ in $\mathfrak{g}$. Then for $i,j \in \{ 1, \ldots , \ell \}$ one has
\[ \exp (v_j) = g_j  = \Adjoint _{g_i} ( g_j)  = \Adjoint _{ g_i} (\exp (v_j) )  = \exp ( \Adjoint _{g_i } ( v_j ) ). \]
Hence $g_i$ commutes with $\exp (t v_j) $ for all $t \in \mathbb{R}$. A similar argument shows 
\[  [ \exp (s v_i) , \exp (t v_j)   ] = e \, \text{ for all } s,t \in \mathbb{R} .  \]
This implies that $\mathfrak{a}$ is a commutative Lie algebra, and the result follows from the abelian case applied to $\exp ( \mathfrak{a} )$.
\end{proof}

Hilbert's fifth problem consists of understanding which locally compact Hausdorff groups are Lie groups. One satisfactory answer is Theorem \ref{thm:mz-hilbert}. On a different direction, Andrew Gleason and Hidehiko Yamabe showed that any locally compact Hausdorff group is not far from being a Lie group \cite{Gleason, Yamabe}.

\begin{theorem}\label{thm:gleason-yamabe}
\rm (Gleason--Yamabe) Let $G$ be a locally compact Hausdorff group. Then there is an open subgroup $\mathcal{O}\leq G$ with the following property: for any open neighborhood of the identity $U \subset \mathcal{O}$, there is a compact normal subgroup $K \triangleleft \mathcal{O}$ with $K \subset U$ such that $\mathcal{O}/K$ is a connected Lie group.
\end{theorem} 

The following result by Victor Glushkov implies that the set of compact normal subgroups with the property that the corresponding quotient is a connected Lie group is closed under finite intersections \cite{Glu}.

\begin{theorem}\label{thm:glushkov}
\rm (Glushkov) Let $\mathcal{O}$ be a locally compact Hausdorff group, and $K_1, K_2$ compact normal subgroups such that both $\mathcal{O}/K_1$ and $\mathcal{O}/K_2$ are connected Lie groups. Then $\mathcal{O}/(K_1 \cap K_2)$ is a connected Lie group.
\end{theorem}

\begin{corollary}\label{cor:glushkov}
\rm Let $G$ be a locally compact group equipped with a left invariant geodesic metric. If $G$ is not a Lie group, then it contains a sequence of compact, normal, non-trivial subgroups $K_1 \geq K_2 \geq \ldots$ with
\begin{equation}\label{eq:glushkov}
\bigcap_{j=1}^{\infty} K_j  = \{ e \}   
\end{equation} 
such that $G/K_j$ is a connected Lie group for all $j$. Moreover, for infinitely many $j$, the identity connected component of $K_{j}/K_{j+1}$  is non-trivial.
\end{corollary}

\begin{proof}
By Theorems \ref{thm:gleason-yamabe} and \ref{thm:glushkov}, we obtain a sequence of compact, normal, non-trivial subgroups $K_1 \geq K_2 \geq \ldots$ satisfying (\ref{eq:glushkov}) and such that $H_j : = G/K_j$ is a connected Lie group for all $j$. 
\begin{center}
\textbf{Claim: }For infinitely many $j$, the identity connected component of $K_{j}/K_{j+1}$  is non-trivial.
\end{center}
Assume by contradiction there is  $j_0\in \mathbb{N}$ such that $K_j/K_{j+1} $ is discrete for all $j \geq j_0$. Let $\delta > 0 $ be small enough so that $B^{H_{j_0}} (e, \delta) $ contains no non-trivial subgroups, and $B^{H_{j_0}} (e, \delta) \to H_{j_0}$ has no content. By Lemma \ref{lem:lift-neighborhood} and induction, the balls $B^{H_j } (e,\delta)$ contain no non-trivial subgroups for $j \geq j_0$. From (\ref{eq:glushkov}), there is $\ell_0 \in \mathbb{N}$ such that for all  $\ell \geq \ell_0$, we have $K_{\ell} \subset B^{G}(e, \delta)$, and 
\[ K_{\ell} / K_{\ell + 1} \subset B^{H_{\ell + 1}  }(e, \delta ) .      \]
This implies $K_{\ell} = K _{\ell + 1}$ for all $\ell \geq \ell_0$, which by (\ref{eq:glushkov}) means $K_{\ell_0} = 0$,  a contradiction.
\end{proof}

\subsection{Gromov--Hausdorff convergence}

\begin{definition}
\rm Let $X, Y$ be metric spaces. A function $f: X \to Y$ is called a \textit{global }$\varepsilon$\textit{-approximation} if 
\begin{itemize}
\item For all $x_1, x_2 \in X$, one has $\vert d(fx_1, fx_2) - d(x_1, x_2) \vert \leq \varepsilon$.
\item For all $y \in Y$, there is $x \in X$ with $d(fx,y) \leq \varepsilon$.
\end{itemize}
\end{definition}

\begin{definition}
\rm Let $(X,p), (Y,q)$ be pointed metric spaces. A function $f: X \to Y$ is called a \textit{pointed }$\varepsilon$\textit{-approximation} if $d(fp,q) \leq \varepsilon$ and
\begin{itemize}
\item For all $x_1, x_2 \in B^X(p, 2/\varepsilon )$, one has $\vert d(fx_1, fx_2) - d(x_1, x_2) \vert \leq \varepsilon$.
\item For all $y \in B^Y(q, 1/\varepsilon )$, there is $x \in B^X(p, 2/\varepsilon )$ with $d(fx,y) \leq \varepsilon$.
\end{itemize}
\end{definition}

\begin{definition}
\rm We say a sequence of pointed proper metric spaces $(X_n, p_n)$ \textit{converges in the pointed Gromov--Hausdorff sense} to a pointed proper metric space $(X,p)$ if there is a sequence of pointed $\varepsilon_n$-approximations $f_n : X_n \to X$ with $\varepsilon_n \to 0$ as $n \to \infty$. 

The functions $f_n$ above are called \textit{Gromov--Hausdorff approximations}. If a sequence $x_n \in X_n$ with $ \sup _{n} d(x_n, p_n) < \infty $  is such that $f_n (x_n) \to x \in X$, by an abuse of notation we say that $x_n$ \textit{converges} to $x$. 
\end{definition}

\begin{definition}\label{def:equivariant}
\rm Let  $(X_n,p_n)$ be a sequence of pointed proper metric spaces converging in the pointed Gromov--Hausdorff sense to a pointed proper metric space $(X,p)$, a sequence of Gromov--Hausdorff approximations $f_n : X_n \to X$, a sequence of closed groups $G_n \leq \iso(X_n)$, and a closed group $G \leq \iso (X)$. Equip $G_n $ with the metric $d_0^{p_n}$ and $G$ with the metric $d_0^p$ from (\ref{eq:d0}). We say that the sequence $G_n$ \textit{converges equivariantly} to $G$ if there is a sequence of Gromov--Hausdorff approximations $\varphi _n : G _ n \to G  $ such that for each $R > 0 $ one has 
\[  \lim_{n \to \infty} \sup_{g \in B^{G_n}(\id_{X_n}, R) } \sup_{ x \in B^{X_n}(p_n,R)}  d( f  _n( g x ), \varphi_n(g) ( f _n x ) ) = 0       . \]
\end{definition}
\begin{remark}\label{rem:equivariant-composition}
\rm Under the conditions of Definition \ref{def:equivariant}, it is easy to see that if $\diam (X_n / G_n) \to 0$, then $G$ acts transitively on $X$, and for all $R>0$ one has
\begin{equation}\label{eq:equivariant-composition}
\lim_{n \to \infty} \sup_{g,h \in B^{G_n}( \id_{X_n} , R )} d _0 ^p(  \varphi_n (  gh ) , \varphi_n(g) \varphi_n(h) ) = 0.
\end{equation}  
\end{remark}
\begin{remark}\label{rem:commutator-passes-to-limit}
\rm If for some $\varepsilon > 0 $, $N \in \mathbb{N}$, one has
\[ G_n^{(N)} \subset    \{ g \in \iso (X_n) \vert d(gx,x) \leq \varepsilon \text{ for all } x \in X_n \}           \]    
for $n$ large enough, then by repeated applications of (\ref{eq:equivariant-composition}), one gets
(cf. \cite[Corollary 2.1.4]{BFT})  
\[ G^{(N)} \subset    \{ g \in \iso (X) \vert d(gx,x) \leq \varepsilon \text{ for all } x \in X \} .  \]
\end{remark}
Isometry groups of proper spaces satisfy a compactness property \cite[Proposition 3.6]{FY92}.
\begin{theorem}\label{thm:equivariant-compactness}
\rm (Fukaya--Yamaguchi) Let  $(X_n,p_n)$ be a sequence of pointed proper metric spaces converging in the pointed Gromov--Hausdorff sense to a pointed proper metric space $(X,p)$, a sequence of Gromov--Hausdorff approximations $f_n : X_n \to X$, and a sequence of closed groups $G_n \leq \iso(X_n)$. Then, after taking a subsequence, the sequence $G_n$ converges equivariantly  to a closed group $G \leq \iso (X)$. 
\end{theorem}

\subsection{Constructing covering spaces}\label{CoverConstruction}

\begin{definition}
\rm Let $\varepsilon > 0$. We say that a covering map $\tilde{X} \to X$  of geodesic spaces is $\varepsilon$\textit{-wide} if for every  $x \in X$, the ball $B^X(x, \varepsilon ) $ is an evenly covered neighborhood of $x$. 
\end{definition}

\begin{definition}
\rm If $(X,p)$ is a pointed geodesic space and $\varepsilon > 0$, we denote by $G(X, \varepsilon ) $ the quotient of $\pi_1(X,p)$ by the (normal) subgroup generated by loops of the form $ \beta \ast  \gamma \ast \overline{\beta }  $, with $\beta (0) = p$, $\beta (1) = \gamma (0) $, and $\gamma$ a loop contained in an open ball of radius $\varepsilon$.
\end{definition}

 The following result is obtained via the standard construction of covering spaces   \cite[Theorem 77.1]{Munk}. 

\begin{lemma}\label{lem:cover-characterization}
\rm Let $(X,p)$ be a pointed proper geodesic space,  $G $ a group, and $\varepsilon > 0$. Then there is a surjective morphism $G(X, \varepsilon ) \to G $ if and only if there is a regular $\varepsilon$-wide covering $\tilde{X} \to X$ with Galois group $G $.
\end{lemma}

The following result by Sormani and Wei implies that if two geodesic spaces are sufficiently close in the Gromov--Hausdorff sense, then one can transfer regular covers from one to the other \cite[Theorem 3.4]{SWHau}.

\begin{theorem}\label{thm:cover-copy}
\rm (Sormani--Wei) Let $X, Y$ be proper geodesic spaces for which there is a global $\varepsilon / 100$-approximation between them. Then there is a surjective map $G(X, \varepsilon / 2 ) \to G(Y, \varepsilon) $. In particular, if there an $\varepsilon$-wide regular cover $\tilde{Y} \to Y$ with Galois group $G$, then there is a surjective morphism $\pi_1(X) \to G$.
\end{theorem}

Now we give a way to construct covering spaces from group actions (cf. \cite[Lemma A.19]{FY92}).

\begin{theorem}\label{thm:monodromy}
\rm  Let $(X,p)$ be a proper geodesic space and $\Gamma \leq \iso (X)$   a discrete group of isometries with $\diam (X/\Gamma) \leq  \rho  $ for some $\rho  > 0$. Define $B:= B(p, 2 \rho  )$ and
\[  S := \{ g \in \Gamma \vert d(gp, p) < 4 \rho  \} = \{ g \in \Gamma \vert gB \cap B \neq \emptyset \} . \]
Let $  \tilde{\Gamma } $ be the abstract group generated by $S$, with relations 
\begin{center}
$s= s_1 s_2$ in $ \tilde{\Gamma } $, whenever $s, s_1, s_2 \in S$ and $s= s_1s_2$ in $\Gamma$. 
\end{center}
If we denote the  canonical embedding $S  \hookrightarrow \tilde{\Gamma} $ as $(s \to s^{\sharp } )$, then there is a unique group morphism $\Phi :  \tilde{\Gamma } \to \Gamma$ that satisfies $\Phi (s^{\sharp }) = s$ for all $s \in S$. Equip $ \tilde{\Gamma } $ with the discrete topology, and consider the topological space
\begin{center}
$\tilde{X} := \left(  \tilde{\Gamma }  \times B \right) / \sim$,
\end{center}
where $\sim$ is the minimal equivalence relation such that 
\begin{equation}\label{eq:x-tilde-def}
 (gs^{\sharp}, x) \sim (g,s x)  \text{ whenever }s \in S, \,  x ,  sx \in B
\end{equation}
We obtain a continuous map $\Psi : \tilde{X} \to X$ given by
\begin{center}
$\Psi  (g, x) : = \Phi (g)(x)      .$
\end{center}
 Then $\Psi$ is a regular $\rho  $-wide covering map with Galois group $\kker (\Phi )$.
\end{theorem}
The proof of Theorem \ref{thm:monodromy} is obtained from a sequence of lemmas.
\begin{lemma}\label{relation-is-good}
\rm Let $(a,x), (b,y ) \in \tilde{\Gamma} \times B$. The following are equivalent:
\begin{itemize}
\item there is $s \in S$ with $b= as^{\sharp}$, $x=sy$.
\item there is $t \in S$ with $a=bt^{\sharp}$, $y=tx$.
\item $(a,x) \sim (b,y)$.
\end{itemize}
\end{lemma}

\begin{proof}
The first two conditions are equivalent by taking $t $ to be $s^{-1}$, and they imply the third one by definition. Using the fact that the first two conditions are equivalent, the third condition implies that there is a sequence $s_1, \ldots , s_k \in S$ such that 
\[(a,x) \sim (as_1^{\sharp} , s_1^{-1}x) \sim  \ldots  \sim (as_1^{\sharp} \cdots s_k^{\sharp} , s_k^{-1} \cdots s_1^{-1} x ) = (b,y). \]
This implies that $ \left( s_1^{-1} \right) , \left(  s_2^{-1}s_1^{-1}  \right) , \ldots , \left(  s_k^{-1}\ldots s_1^{-1} \right) \in S$, allowing us to prove by induction on $j$ that $(s_1\cdots s_j)^{\sharp} s_{j+1}^{\sharp} = (s_1 \cdots s_{j+1}  )^{\sharp}$ in $\tilde{\Gamma }$. This implies the first condition by taking $s$ to be $(s_1 \cdots s_k) \in S$. 
\end{proof} 
Fix $U \subset X$ an  open ball of radius $\rho  $. Since $\Phi$ is surjective and $\diam (X/\Gamma ) \leq \rho  $, there is $g_0 \in \tilde{\Gamma}$ such that 
\[V : = \Phi(g_0^{-1})(U) \subset B^X(p, 2 \rho  ).\]
\begin{lemma}\label{lem:monodromy-i}
\rm  The preimage of $U$ is given by
 \begin{equation}\label{eq:preimage-u}
 \Psi ^{-1}(U) = \bigcup_{g \in \kker(\Phi)} \left( \bigcup_{s \in S} \left( \{ g_0gs^{\sharp}  \} \times \left(( s^{-1}V) \cap B   \right)     \right)     \right)  / \sim   .
\end{equation}
\end{lemma}
\begin{proof}
By direct evaluation, if  $g \in \kker (\Phi )$ and  $x \in V \cap sB$, then 
\[  \Psi (g_0gs^{\sharp}, s^{-1}x) = \Phi (g_0g)x = \Phi(g_0) x \in \Phi (g_0) (V) = U. \]
On the other hand, if $(h,x) \in \tilde{\Gamma} \times B$ is such that $\Phi(h) (x) \in U$, then $\Phi(g_0^{-1}h)(x) \in V$, and $\Phi(h) \in \Phi( g_0) S $, implying that $h = g_0gs^{\sharp}$ for some $g \in \kker (\Phi)$, $s \in S$. Also, $s(x) = \Phi ( g_0^{-1} g_0 s^{\sharp} )(x) = \Phi(g_0^{-1}h)(x) \in V$, proving that the class of $(h,x)$ belongs to the right hand side of (\ref{eq:preimage-u}).
\end{proof}
We say that a subset $A \subset \tilde{\Gamma } \times B$ is \textit{saturated} if it is a union of equivalence classes of the relation $\sim$. 
\begin{lemma}\label{lem:monodromy-ii}
\rm As $g$ ranges through $ \kker (\Phi )$, the sets 
\[   W_g : =   \bigcup_{s \in S} \left( \{ g_0gs^{\sharp}  \} \times \left(( s^{-1}V) \cap B   \right)     \right)     \subset \tilde{\Gamma } \times B    \]
are open, disjoint, and saturated.
\end{lemma}
\begin{proof}
 The fact that they are open is straightforward, since $(s^{-1}V) \cap B$ is open in $B$ for each $s \in S$. To prove that they are disjoint, assume that 
\[  (g_0g_1s_1^{\sharp},x  ) \sim (g_0 g_2 s_2 ^{\sharp} , y)     \]
  for some  $g_1, g_2 \in \kker(\Phi),$ $s_1s_2 \in S , $ $ x\in (s_1^{-1}V) \cap B , y \in (s_2^{-1}V ) \cap B . $ Lemma \ref{relation-is-good} implies that there is $t \in S$ with 
  \begin{equation}\label{eq:mono}
  g_0g_1s_1^{\sharp} t^{\sharp} = g_0g_2s_2^{\sharp}  , \text{ } x=ty . 
  \end{equation}
 By Taking $\Phi$ on both sides of the first equation, we get $s_2 = s_1t$, and hence $s_2^{\sharp} = s_1^{\sharp}t^{\sharp}$. Canceling this in (\ref{eq:mono}), we get $g_1=g_2$, proving that the sets $W_g$ are disjoint.

To prove that they are saturated, assume that for some $(h,x ) \in \tilde{\Gamma} \times B$, $g \in \kker (\Phi) $, $s \in S$, $y \in (s^{-1} V )\cap B$, we have $(h,x) \sim (g_0gs^{\sharp} , y)$. By Lemma \ref{relation-is-good}, there is $t \in S$ with 
\[  h = g_0gs^{\sharp}t^{\sharp}, \text{ } y = tx.        \]
This implies that $st (x)= s(y) \in V \subset  B$, hence $st \in S$ and $(st)^{\sharp}=s^{\sharp}  t ^{\sharp} $. Then 
\[ (h,x) \in  \{ g_0 g (st)^{\sharp} \} \times  \left( ((st)^{-1}V) \cap B \right) , \]
proving that $W_g$ is saturated.
\end{proof}
\begin{lemma}\label{lem:monodromy-iii}
\rm For each $g \in \kker (\Phi )$, the image of $W_g$ in $\tilde{X}$ is sent  homeomorphically via $\Psi$ onto $U$.
\end{lemma}
\begin{proof}
Surjectivity is straightforward, since $\Phi ( g_0 g ) (V) = U$. To check injectivity, assume that for some $s,t \in S$, $x\in (s^{-1}V) \cap B$, $y \in (t^{-1}V) \cap B$, we have
\[ \Phi(g_0gs^{\sharp})(x) = \Phi (g_0gt^{\sharp})(y).      \]
Then $x=s^{-1}ty \in B$, implying that $s^{-1}t \in S$, and consequently, $(s^{-1}t)^{\sharp} = (s^{\sharp})^{-1}t^{\sharp}$. Hence
\[   g_0gs^{\sharp} (s^{-1}t)^{\sharp} = g_0gt^{\sharp} , \text{ }x= (s^{-1}t)y,  \]
obtaining injectivity. To check that $\Psi \vert_{W_g}$ is open, take $\mathcal{O} \subset W_g$ open and saturated containing the class of 
\[(g_0g, x) \in  \{g_0g\} \times  V   . \]
Then $\Psi $ sends $ ( \{g_0g\} \times  V    ) \cap \mathcal{O} $ to an open neighborhood of $\Phi(g_0)(x)$. Since $(g_0g,x) $ was arbitrary, $\Psi \vert _{W_g} $ is open.
\end{proof}

\begin{proof}[Proof of Theorem \ref{thm:monodromy}:] 
By Lemmas \ref{lem:monodromy-i}, \ref{lem:monodromy-ii}, and \ref{lem:monodromy-iii}, $U$ is an evenly covered neighborhood. Since $U$ was arbitrary, $\Psi $ is a $\rho$-wide covering map.

From (\ref{eq:x-tilde-def}), one sees that left multiplication on $\tilde{\Gamma }$ on itself descends to an action $\Theta : \tilde {\Gamma } \to \iso (\tilde{X})$. For $g \in \kker (\Phi)$, $h \in \tilde {\Gamma }$, $x \in B$, we have
\[    \Psi( \Theta (g) (h,x) ) = \Psi ( gh,x ) = \Phi (gh)x = \Phi (h) x = \Psi (h,x), \]
so $\Theta \vert _{\kker (\Phi)}$  consists of deck transformations. From Lemma \ref{lem:monodromy-ii}, it is straightforward to check that for each $g_0 \in \tilde{\Gamma }$, the group $\kker (\Phi) $ acts on the set of $W_g$'s freely and transitively. This shows that $\Psi$ is regular and $\kker (\Phi )$ is its Galois group, finishing the proof of Theorem  \ref{thm:monodromy}.
\end{proof}

\subsection{Local groups}

In this section we present the elementary theory of local groups and approximate groups we will use. We refer the reader to \cite[Appendix B]{BGT}  for proofs and further discussion.

\begin{definition}
\rm Let $(G,e)$ be a pointed topological space. We say that $G$ is a \textit{local group}, if there are continuous maps $( )^{-1}: G \to G$ and $\cdot : \Omega  \to G$ for some $ \Omega \subset G \times G$ such that
\begin{itemize}
\item $\Omega$ is an open set containing $(G \times \{ e\} ) \cup ( \{ e \} \times G  )  $.
\item For all $g \in G$, we have  $g \cdot e = e \cdot g = g$.
\item For all $g \in G$, we have $(g,g^{-1}), (g^{-1},g)\in \Omega$ and $g \cdot g^{-1} = g^{-1} \cdot g = e$.
\item For all $g,h,k \in G$ such that $(g,h) , (gh,k), (h,k), (g, hk) \in \Omega$, we have $g (hk)= (gh)k$.
\end{itemize}
\end{definition}

\begin{definition}
\rm We say that a local group $G$ is a \textit{local Lie group} if it is a smooth manifold, and the maps $( )^{-1}: G \to G$ and $\cdot : \Omega  \to G$ are smooth.
\end{definition}

\begin{definition}
\rm Let $G$ be a local group. We say that a subset $A \subset G$ is \textit{symmetric} if $e \in A$ and $g^{-1} \in A$ for all $g \in A$. 
\end{definition}

\begin{definition}
\rm Let $G$ and $H$ be two local groups. We say a continuous function $\phi : G \to H$ is a morphism if the following holds:
\begin{itemize}
\item $\phi (e) = e$.
\item For all $g \in G$, we have $\phi (g^{-1}) = \left[ \phi (g)   \right] ^{-1}$.
\item If $g,h \in G$ are such that $g\cdot h$ is defined in $G$, then $\phi (g) \cdot \phi (h) $ is defined in $H$, and $\phi (g\cdot h ) = \phi (g) \phi (h)$. 
\end{itemize}
\end{definition}

\begin{definition}
\rm Let $(G,e$) be a local group, and $U\subset G$ a symmetric subset. We say that $U$ is a \textit{restriction} of $G$ if when restricting $( )^{-1} $ to $U$, and $\cdot : \Omega \to G$ to $ \{ (g,h) \in \Omega \cap U \times U \vert gh \in U \}  $, we obtain a local group structure on $U$.
\end{definition}

\begin{definition}
\rm Let $G$ be a local group, and $g_1, \ldots , g_m \in G$. We say that the product $g_1 \cdots g_m$ is \textit{well defined in} $G$ if for each $1 \leq j\leq k \leq m$, we can find an element $g_{[j,k]} \in G$ such that  
\begin{itemize}
\item For all $j \in \{ 1, \ldots , m \} $, we have $g_{[j,j]}= g_j$ 
\item For all $1 \leq j \leq k < \ell \leq m$, the pair $\left(  g_{[j,k]} , g_{[k+1,\ell ]}  \right) $ lies in $\Omega$, and $g_{[j,k]} \cdot g_{[k+1,\ell ]} = g_{[j, \ell ]}$. 
\end{itemize}
For sets $A_1, \ldots , A_m \in G$, we say that the product $A_1\cdots A_m$ is \textit{well defined} if for all choices of $g_j \in A_j$, the product $g_1 \cdots g_m$ is well defined.
\end{definition}
\begin{definition}
\rm Let $G$ be a local group. We say that a subset $A \subset G$ is a \textit{multiplicative set} if it is symmetric, and $A^{200}$ is well defined in $G$.
\end{definition}
\begin{definition}
\rm We say that a local group $G$ is  \textit{cancellative} if the follwing holds:
\begin{itemize}
\item For all $g,h,k \in G$ such that $gh$ and $gk$ are well defined and equal to each other, we have $h=k$.
\item For all $g,h,k \in G$ such that $hg$ and $kg$ are well defined and equal to each other, we have $h=k$.
\item For all $g,h \in G$ such that $gh $ and $h^{-1}g^{-1}$ are well defined, then $(gh)^{-1} = h^{-1}g^{-1}$.
\end{itemize} 
\end{definition}

\begin{definition}
\rm Let $G, G^{\prime}$ be local groups such that $G^{\prime}$ is a restriction of $G$. We say that $G^{\prime}$ is a \textit{sub-local group} of $G$ if there is an open set $V\subset G$ containing $G^{\prime}$ with the property that for all $a,b \in G^{\prime}$ such that $ab$ is well defined in $V$, then $ab \in G^{\prime}$.  If $V$ also satisfies that for all $a \in G^{\prime} $, $b \in V$ such that $bab^{-1}$ is well defined in $V$, then $bab^{-1}\in G^{\prime}$, we say that $G^{\prime}$ is a \textit{normal} sub-local group of $G$, and $V$ is called a \textit{normalizing neighborhood} of $G^{\prime}$.
\end{definition}

\begin{lemma}\label{lem:local-quotient}
\rm Let $G$ be a cancellative group and $H$ be a normal sub-local group with normalizing neighborhood $V$. Let $W \subset G$ be an open symmetric subset such that $W^6 \subset V$. Then there is a cancellative local group $W/H $ equipped with a surjective morphism $\phi : W \to W / H$ such that, for all $g,h \in W$, one has $\phi (g) = \phi (h)$ if and only if $gh^{-1}\in H$, and for any $E \subset W / H$, one has that $E$ is open if and only if $\phi^{-1}(E)$ is open. 
\end{lemma}

\begin{definition}\label{def:a-group}
\rm Let $A$ be a finite symmetric subset of a  multiplicative set and $C \in \mathbb{N}$. We say that $A$ is a $C$\textit{-approximate group} if $A^2$ can be covered by $ C$ left translates of $A$.
\end{definition}

\begin{definition}\label{def:strong-a-group}
\rm Let $A$ be a $C$-approximate group for some $C \in \mathbb{N}$. We say that $A$ is a  \textit{strong }$C$\textit{-approximate group}  if there is a symmetric set $S \subset A$ satisfying the following:
\begin{itemize}
\item $\left( \{ asa^{-1} \vert a \in A^4, s \in S\}  \right) ^{10^3C^3} \subset A$. 
\item If $g, g^2, \ldots , g^{1000} \in A^{100}$, then $g \in A$.
\item If $g, g^2 , \ldots , g^{10^6 C^3} \in A$, then $g \in S$.
\end{itemize}
\end{definition}  

\begin{definition}
\rm Let $A $ be a subset of a multiplicative set $G$. For $g\in G$, we define the escape norm as
\begin{center}
$\Vert g \Vert_{A} : = \inf \left\{  \dfrac{1}{m+1} \bigg| \text{ } e, g, g^2, \ldots , g^m \in A   \right\}$.  
\end{center}
\end{definition} 
In strong approximate groups, the escape norm satisfies really nice properties \cite[Theorem 8.1]{BGT}. 
\begin{theorem} \label{thm:gleason}
 \rm (Gleason--Breuillard--Green--Tao) For each $C>0$, there is $M>0$ such that if $A$ is a strong $C$-approximate group and $g_1, g_2, \ldots , g_k \in A^{10}$,  then
\begin{enumerate}[label=\roman*]
\item $\Vert g_1g_2 \cdots g_k \Vert_A \leq M \sum_{j=1}^k \Vert  g_j  \Vert_A $.\label{thm:gleason-product}
\item $\Vert g_2g_1g_2^{-1} \Vert_A \leq 10^3 \Vert g_1 \Vert_A$.\label{thm:gleason-conjugate}
\item $\Vert [g_1,g_2]\Vert _A \leq M \Vert  g_1 \Vert_A \Vert g_2 \Vert_A $.\label{thm:gleason-commutator}
\end{enumerate}
\end{theorem}

\begin{definition}
 \rm Let $G$ be a local group, $u_1, u_2, \ldots , u_r \in G$, and $N_1, N_2, \ldots , N_r $ $\in \mathbb{R}^+$. The set $P(u_1, \ldots , u_r ; N_1, \ldots , N_r)$ is defined as the set of words in the $u_i$'s and their inverses such that the number of appearances of $u_i$ and $u_i^{-1}$ is not more than $N_i$. We say that  $P(u_1, \ldots , u_r ; N_1, \ldots , N_r)$ is well defined if every word in it is well defined in $G$. When that is the case, we call it a \textit{progression of rank} $r$ (a progression of rank $0$ is defined to be the trivial subgroup). We say a progression $P(u_1, \ldots , u_r ; N_1, \ldots , N_r)$  is a \textit{nilprogression in $C$-normal form} for some $C>0$ if it also satisfies the following properties:
\begin{enumerate}[label=\textbf{N.\arabic*}]
\item  For all $1 \leq i \leq j \leq r$, and all choices of signs, we have
\begin{center}
$  [ u_i^{\pm 1} , u_j^{\pm 1} ]  \in P \left(  u_{j+1} , \ldots , u_r ; \dfrac{CN_{j+1}}{N_iN_j} , \ldots ,   \dfrac{CN_r}{N_iN_j}  \right).     $
\end{center}\label{condition:n1}
\item The expressions $ u_1 ^{n_1} \ldots u_r^{n_r} $ represent distinct  elements as $n_1, \ldots , n_r$ range over the integers with  $\vert n_1 \vert \leq  N_1/C , \ldots , \vert n_r \vert \leq  N_r/C$.\label{condition:n2}
\item One has 
\[  \frac{1}{C} (2\lfloor N_1 \rfloor + 1 ) \cdots ( 2\lfloor N_r \rfloor + 1 ) \leq \vert P \vert  \leq C (2 \lfloor N_1 \rfloor + 1 ) \cdots ( 2\lfloor N_r \rfloor + 1 ) .               \]\label{condition:n3}
\end{enumerate} 
For a nilprogression $P$ in $C$-normal form, and $\varepsilon \in (0,1)$, the progression  $P( u_1, \ldots , u_r ; $ $ \varepsilon N_1, \ldots , \varepsilon N_r   )$ also satisfies \ref{condition:n1} and \ref{condition:n2}, and we denote it by $\varepsilon P$. We define the \textit{thickness} of $P$ as the minimum of $N_1, \ldots , N_r$ and we denote it by $ \thi (P)$. The set $\{ u_1 ^{n_1 }\ldots u_r^{n_r} \vert \vert n_i \vert \leq N_i/C  \}$ is called the \textit{grid part of }$P$, and is denoted by $G(P)$.
\end{definition}

\begin{definition}
\rm Let $P(u_1 , \ldots , u_r ; N_1 , \ldots , N_r)$ be a nilprogression in $C$-normal form with  $\thi (P )\geq C$. Set $\Gamma_P$ to be the abstract group generated by $\gamma_1, \ldots , \gamma_r$ with relations $[\gamma _j , \gamma_k]  =   \gamma_{k+1}^{\beta_{j, k}^{ k+1}} \ldots \gamma_r^{\beta_{j, k}^r} $ whenever $j < k $, where $[u _j , u_k]  =   u_{k+1}^{\beta_{j, k}^{ k+1}} \ldots u_r^{\beta_{j, k}^{ r}} $ and $\vert \beta_{j, k}^{l } \vert \leq \frac{CN_l}{N_j N_k}$. We say that $P$ is \textit{good} if  each element of $\Gamma _P$ has a unique expression of the form
\begin{center}
$ \gamma_1^{n_1}\ldots \gamma_r^{n_r},  $ with $n_1, \ldots, n_r \in \mathbb{Z}$.
\end{center}
\end{definition}

\begin{theorem}\label{thm:malcev-embedding}
\rm  (Malcev) For each $r \in \mathbb{N}$, $C > 0 $, there is $\varepsilon > 0 $ such that the following holds. Let $P(u_1 , \ldots , u_r ;$ $ N_1 , \ldots , N_r)$ be a nilprogression in $C$-normal form. If $\thi (P)$ is large enough depending on $r$ and $ C$, then $P$ is good and the map $u_j \to \gamma_j$ extends to a product preserving embedding $ \sharp : G(\varepsilon P) \to \Gamma_P.$ For $A \subset G(\varepsilon P)$, we will denote its image under this embedding by $A^{\sharp}$.
  Furthermore, there is a quasilinear polynomial group structure (see Definition \ref{def:quasil})
\[ Q: \mathbb{R}^{r}\times \mathbb{R}^r \to \mathbb{R}^r \]
of degree $\leq d(r)$ such that the multiplication in $\Gamma_P$ is given by 
\begin{center}
$ \gamma_1^{n_1}\ldots \gamma_r^{n_r} \gamma_1^{m_1}\ldots \gamma_r^{m_r} = \gamma_1^{(Q(n,m))_1}\ldots \gamma_r^{(Q(n,m))_r}   $ for $n,m \in \mathbb{Z}^r,$
\end{center}
so $\Gamma_P$ is isomorphic, via $\gamma_j \to  e_j$, to the lattice $(\mathbb{Z}^r, Q\vert_{\mathbb{Z}^r \times \mathbb{Z}^r} )$. $Q$ is called the \textit{Malcev polynomial} of $P$, and $(\mathbb{R}^r, Q)$ the \textit{Malcev Lie group} of $P$. 
\end{theorem}

\begin{proof}
In \cite[Lemma C.3]{BGT} it is shown that provided $\thi (P)$ is large enough, $G(\varepsilon P)$ embeds in a product preserving way to $\Gamma_P$ (see also \cite[Section 4.6]{BK}). In \cite[Section 5.1]{BK} it is shown that if $\thi (P)$ is large enough, the product in $\Gamma_{P}$ is given by a polynomial $Q$ that extends to the desired group structure in $\mathbb{R}^r$.
\end{proof}

\subsection{Ultralimits}

In this section we discuss the ultrafilter tools we will use during the proof of Theorem \ref{thm:part-iii}. We refer the reader to \cite[Section 2.1]{BFT} and  \cite[Appendix A]{BGT} for proofs and further discussions.

\begin{definition}
 \rm Let $\wp (\mathbb{N})$ denote the power set of the natural numbers and $\alpha : \wp (\mathbb{N}) \to \{ 0,1 \}$ a function. We say that $\alpha$ is a \textit{non-principal ultrafilter} if it satisfies:
\begin{itemize}
\item $\alpha (\mathbb{N}) = 1$.
\item $\alpha (A \cup B )= \alpha (A) + \alpha (B)$ for all disjoint $A , B  \subset \mathbb{N}$.
\item $\alpha (F )=0$ for all finite $F \subset \mathbb{N}$.
\end{itemize}
\end{definition}
 Using Zorn's Lemma it is not hard to show that non-principal ultrafilters exist. We will choose one $(\alpha )$ and fix it for the rest of this paper.  For a property $P : \mathbb{N}\to \{0,1\}$, if $\alpha (  P^{-1} (1)  ) =1$ we  say that ``$P(n)$ holds for $\alpha$-large enough $n$''. 

 \begin{definition}
 \rm Let $A_n$ be a sequence of sets. In the product 
\[A^{\prime} : = \prod _{n=1}^{\infty} A_n,\]
we say that two sequences $\{ a_n \} _n , \{  a_n^{\prime}   \} _n$ are $\alpha $\textit{-equivalent} if 
\[ \alpha \left( \{  n \vert a_n = a_n^{\prime}   \} \right) =1.  \]
The set $A^{\prime}$ modulo this equivalence relation is called the \textit{ultraproduct} of the sets $A_n$ and is denoted by
\[ \textbf{A}  :  =   \lim\limits_{n \to \alpha} A_n.      \]
If the sets $A_n$ are local groups, then $\textbf{A}$ inherits a local group structure given by $\{a_n \}_n \ast \{ a_n^{\prime} \} _n : = \{ a_n a_n^{\prime } \}_n $, whenever $ \{ a_na_n^{\prime } \}_n \in \textbf{A}$.
\end{definition}

\begin{definition}
\rm If $A_n = \mathbb{R}$ for each $n$, then $\textbf{A}$ is denoted by ${}^{\ast}\mathbb{R}$ and its elements are called \textit{non-standard real numbers}. Elements of $\mathbb{R}$ are called \textit{standard real numbers}, and there is a natural embedding $\mathbb{R} \hookrightarrow {}^{\ast} \mathbb{R}$ whose image consists of the constant sequences.

For $x= \{ x_n\}_n$, $y=\{ y_n \}_n$ non-standard real numbers, we say 
\begin{itemize}
\item $x \leq y$ if $\alpha \left( \{ n \in \mathbb{N} \vert x_n \leq y_n \}   \right) =1$.
\item $x = O (y)$ if there is $C \in \mathbb{R}$ such that $x \leq Cy$.
\item $x = o(y)$ if for all $c \in \mathbb{R}$ we have $x \leq cy$.  
\item $x$ is \textit{bounded} if $x = O (1)$.
\item $x$ is \textit{infinitesimal} if $x = o (1)$.
\end{itemize}
\end{definition}

\begin{definition}
\rm   Let $ x_n  $ be a sequence in a metric space $X$. We say that the sequence \textit{ultraconverges} to a point $x_{\infty} \in X$ if for every $\varepsilon > 0$, 
\[ \alpha \left(  \{ n  \vert d(x_n, x_{\infty}) < \varepsilon  \}  \right) = 1  .   \] 
If this is the case, the point $x_{\infty} $ is called the \textit{ultralimit} of the sequence, and we write  $x_n \xrightarrow{\alpha} x_{\infty}$ or
\[ \lim\limits_{n \to \alpha} x_n = x_{\infty}. \]
\end{definition}
It is easy to show that if a sequence has an ultralimit, then it is unique. Furthermore, if $X$ is compact, then any sequence in $X$ ultraconverges.

\begin{definition}\label{def:ultra-strong}
\rm  Let $A_n$ be a sequence of finite multiplicative sets. If there is a $C \in \mathbb{N}$ such that $A_n$ is a (strong) $C$-approximate group for $\alpha$-large enough $n$, we say that the ultraproduct $\textbf{A}= \lim_{n \to \alpha} A_n $ is a \textit{(strong) ultra approximate group}. If for $\alpha$-large enough $n$, the approximate group $A_n$ does not contain non-trivial subgroups, we say that $\textbf{A}$ is an \textit{NSS} (no small subgroups) ultra approximate group.
\end{definition}
\begin{definition}
\rm  For subsets $A^{\prime}_n \subset A_n^4$ with the property  $(A^{\prime}_n)^4 \subset A_n^4$, we say that the ultraproduct $\textbf{A}^{\prime}=\lim_{n \to \alpha}   A^{\prime}_n $ is a \textit{sub-ultra approximate group} of $\textbf{A}$ if it is an ultra approximate group, and there is a constant $C^{\prime}\in \mathbb{N}$ such that $A_n$ can be covered by $C^{\prime}$ many translates of $A_n^{\prime}$ for $\alpha$-large enough $n$.
\end{definition}
\begin{definition}\label{def:model}
\rm Let $\textbf{A}=\lim_{n \to \alpha} A_n$ be an ultra approximate group. A \textit{good Lie model} for $\textbf{A}$ is a connected local Lie group $\Gamma $, together with a morphism $\sigma: \textbf{A}^8 \to \Gamma $ satisfying:
\begin{enumerate}[label=\textbf{M.\arabic*}]
\item The image $\sigma (\textbf{A} ) \subset \Gamma $ is precompact.\label{def:model-pre-compact}
\item There is an open neighborhood $U_0 \subset \Gamma $ of the identity with $U_0 \subset \sigma (\textbf{A}) $ and $\sigma^{-1}(U_0) \subset \textbf{A}$.\label{def:model-u0}
\item For $F \subset U \subset U_0$ with $F $ compact and $U$ open, there is an ultraproduct $\textbf{A}^{\prime}= \lim_{n \to \alpha}  A_n^{\prime}$ of finite sets $A_n^{\prime} \subset A_n$ with $\sigma^{-1}(F ) \subset \textbf{A}^{\prime} \subset \sigma^{-1}(U)$.\label{def:model-approximation}
\end{enumerate} 
\end{definition}

\begin{definition}
\rm  Let $P_n$ be a sequence of sets. If for $\alpha$-large enough $n$, $P_n$ is a nilprogression of rank $r$ in $C$-normal form for some $r \in \mathbb{N}$, $C  > 0 $, independent of $n$, we say that the ultraproduct $ \textbf{P} = \lim_{n \to \alpha}  P_n$ is an \textit{ultra nilprogression of rank $r$ in $C$-normal form}. We denote $\lim_{n \to \alpha}  \varepsilon P_n$ as $\varepsilon \textbf{P}$. If $(\thi (P_n))_n$ is unbounded, we say that $\textbf{P}$ is a \textit{non-degenerate} ultra nilprogression. The ultraproduct $G(\textbf{P}) :=\lim_{n \to \alpha}  G(P_n)$ is called the \textit{grid part of }$\textbf{P}$.
\end{definition}

\subsection{Ultraconvergence of polynomials}

\begin{definition}
\rm  Let $  Q_n : \mathbb{R}^{k} \to \mathbb{R}^{\ell} $ be a sequence of polynomials of bounded degree. We say that the sequence converges \textit{well} to a polynomial $Q: \mathbb{R}^{k} \to \mathbb{R}^{\ell}$ if the sequences of coefficients of $Q_n$ ultraconverge to the corresponding coefficients of $Q$.
\end{definition}

\begin{lemma}\label{lem:well-finite}
\rm  For each $d\in \mathbb{N}$, there is $N_0 \in \mathbb{N}$ such that the following holds. Let $I_{N_0}:= \{ -1, \ldots,\frac{-1}{N_0}, 0, \frac{1}{N_0}, \ldots , 1 \} $, and assume we have polynomials $Q_n, Q: \mathbb{R}^{k} \to \mathbb{R}^{\ell}$ of degree $\leq d$ such that $Q_n(x) \xrightarrow{\alpha} Q(x) $ for all $x\in \left( I_{N_0} \right)^{\times k}$. Then $Q_n$ converges well to $Q$. 
\end{lemma}

\begin{proof}
Working on each coordinate, we may assume that $\ell =1$. We proceed by induction on $k$, the case $k =1$ being elementary Lagrange interpolation. Name the variables $x_1, \ldots , x_{k}$. Since
\[ \mathbb{R}[x_1, \ldots , x_{k}]  = (\mathbb{R}[x_1])[x_2, \ldots , x_{k }],  \]
 we can consider the polynomials $Q_n, Q$, as polynomials $\tilde{Q}_n, \tilde{Q}$, in the variables $x_2, \ldots, x_{k}$ with coefficients in $\mathbb{R}[x_1]$. 

If $Q_n(x) \xrightarrow{\alpha} Q(x) $ for all $x\in \left( I_{N_0}\right)^{\times k}$, we would have $\tilde{Q}_n(q, x^{\prime}) \xrightarrow{\alpha} \tilde{Q}(q,x^{\prime})$ for all $q \in I_{N_0}$ and $x^{\prime} \in  (  I_{N_0})^{\times (k -1)}$. By the induction hypothesis, if $N_0$ was large enough, depending on $d$, the coefficients of $\tilde{Q}_n $, which are polynomials in $\mathbb{R}[x_{1}]$, ultraconverge to the coefficients of $\tilde{Q}$ whenever $x_{1}\in I_{N_0}$. By the case $k =1$, if $N_0$ was large enough, the coefficients of $Q_n$ ultraconverge to the coefficients of $Q$.
\end{proof}

\begin{lemma}\label{lem:well-algebra}
\rm Let $Q_n: \mathbb{R}^r \times \mathbb{R}^r \to \mathbb{R}^r$ be a sequence of polynomial group structures in $\mathbb{R}^r$ of bounded degree. Assume $Q_n $ converges well to a polynomial group structure $Q: \mathbb{R}^r \times \mathbb{R}^r  \to \mathbb{R}^r$. Then the corresponding sequence of Lie algebra structures on $\mathbb{R}^r$ converges well to the Lie algebra structure of $Q$.
\end{lemma}

\begin{proof}
This follows from the fact that the structure coefficients of the Lie algebras depend continuously on the derivatives of $Q_n$, which by hypothesis, ultraconverge to the corresponding derivatives of $Q$.
\end{proof}

\begin{definition}\label{def:quasil}
\rm For $x = (x_1, \ldots , x_r) \in \mathbb{R}^r$, we define its \textit{support} as
\[   \text{supp} ( x ) := \{ i \in \{ 1, \ldots , r \} \vert x_i \neq 0     \} . \]
For $x,y \in \mathbb{R}^r$, we say that $x \preceq y$ if $i \leq j$ for every $i \in \text{supp} (x) $, $j \in \text{supp} (y) $. We say that a polynomial group structure $Q : \mathbb{R}^r \times \mathbb{R}^r\to \mathbb{R}^r$ is \textit{quasilinear} if 
\[ Q(x,y) = Q(x, 0) + Q(0,y) = x+y\]
when $x \preceq y$.
\end{definition}

Note that for any quasilinear group structure $Q:\mathbb{R}^r \times \mathbb{R}^r \to \mathbb{R}^r$, the coordinate axes are one-parameter subgroups, and the exponential map 
\[ \exp : \Tan _0 \mathbb{R}^r = \mathbb{R}^r \to \mathbb{R}^r \]
 is the identity when restricted to such axes. Moreover, by the Baker--Campbell--Hausdorff formula, for $x= (x_1, \ldots , x_r) \in \mathbb{R}^r$, the expression
 \begin{align*}
 \log (  x  ) & =   \log \left(  x_{ 1} e_1 + \ldots + x_{i}e_r     \right) \\
 & =  \log \left(    ( x_{ 1} e_1 ) \ldots ( x_{r}e_r)     \right) \\
  & =  \log   (   \exp (  \log  (  x_{1}e_1  ) ) \ldots   \exp( \log (   x_{r}e_r )   )   )      \\
 & =  \log  (  \exp  (x_{1 }e_1)  \ldots \exp( x_{r}e_r) )   
 \end{align*}
is a polynomial on the variables $x_1, \ldots , x_r$, and its coefficients depend continuously on the structure coefficients of the Lie algebra associated to $(\mathbb{R}^r,  Q)$. This, together with Lemma \ref{lem:well-algebra}, implies the following result.

\begin{lemma}\label{lem:log-continuous}
\rm Consider quasilinear polynomial nilpotent group structures $Q_n , Q : \mathbb{R}^r\times \mathbb{R}^r \to \mathbb{R}^r$ of bounded degree. Let $\log_n ,\log : \mathbb{R}^r \to \mathbb{R}^r = \Tan _0 \mathbb{R}^r$ denote the logarithm maps for the group structures $Q_n$ and $Q$, respectively. Assume the sequence $Q_n$ converges well to $Q$, and a sequence $x_n \in \mathbb{R}^r$ ultraconverges to a point $x \in \mathbb{R}^r$. Then
\[ \lim\limits_{n \to \alpha} \log_n (x_n) = \log (x).  \]
\end{lemma}

\section{Nilpotent groups of isometries}\label{sec:part-i}

In this section, we begin the proof of Theorem \ref{thm:main-theorem} with the following result.

\begin{theorem}\label{thm:part-i}
\rm   Let $(X_n,p_n)$ be a sequence of almost homogeneous spaces that converges in the pointed Gromov--Hausdorff sense to a space $(X,p)$. Then $X$ is isometric to a nilpotent locally compact group equipped with an invariant metric.
\end{theorem}

 By hypothesis, there are discrete groups of isometries $G_n \leq \iso (X_n)$ with $\diam (X_n/G_n) \to 0$. To prove Theorem \ref{thm:part-i} we first reduce it to the case when the groups $G_n$ are almost nilpotent.

\begin{lemma}\label{lem:almost-nilpotent}
\rm Under the hypotheses of Theorem \ref{thm:part-i}, there are discrete groups of isometries $G_n^{\prime } \leq \iso (X_n)$ with $\diam (X_n/G_n^{\prime})\to 0$, satisfying that for each $\varepsilon > 0$, there is $N = N(\varepsilon ) \in \mathbb{N} $ such that 
\[    ( G_n ^{\prime})^{(N)} \subset \{ g \in G_n ^{\prime } \vert d(gx,x) \leq \varepsilon \text{ for all } x \in X_n \}           \]
for $n$ large enough.
\end{lemma}

This Lemma is a consequence of the following result, which is one of the strongest versions of the Margulis Lemma. It states that if a sufficiently large ball in a Cayley graph can be covered by a controlled number of balls of half its radius, then the corresponding group is virtually nilpotent.

\begin{theorem}\label{thm:margulis}
\rm (Breuillard--Green--Tao) For $C\in \mathbb{N}$, there is $N(C) \in \mathbb{N}$ such that the following holds: Let $A$ be a finite symmetric subset of a group $G$, which is in turn generated by a finite symmetric set $S$. If $S^N\subset A$, and $A$ is a $C$-approximate group, then there is a subgroup $ G^{\prime} \leq G$ with
\begin{itemize}
\item $[G:G^{\prime}]\leq N$
\item $\left( G^{\prime}\right)^{(N)} \subset A^4 $
\end{itemize} 
\end{theorem}

\begin{proof}
By \cite[Corollary 11.2]{BGT} combined with \cite[Remark 11.4]{BGT}, if $N (C)$ is large enough, $S^N \subset A$, and $A$ is a $C$-approximate group, then there are subgroups $F \triangleleft G^{\prime } \leq G$ with $F \subset A^4$, $[G:G^{\prime } ] \leq N$, such that $G^{\prime}/ F$ is nilpotent of step $\leq N$. Since $(G^{\prime })^{(N)}F / F = (G^{\prime}/F)^{(N)} = \{ e _{G^{\prime} / F} \} $, the result follows. 
\end{proof}

\begin{proof}[Proof of Lemma \ref{lem:almost-nilpotent}:] Fix $k \in \mathbb{N}$.   Since $X$ is proper, there is $C \in \mathbb{N} $ such that $B(p,3/k)$ can be covered by $C$ balls of radius $1/k$. That is, there are $\{ x_1, \ldots , x_{C} \} \in X$ such that 
\begin{equation}\label{eq:doubling-0}
   B(p , 3 /k) \subset \bigcup_{j=1}^{C}  B(x_j, 1/k).  
\end{equation}
Since $\diam (X_n/G_n) \to 0$, for $j \in \{ 1, \ldots , C \} $ there are $g_{j,n} \in G_n$ such that $g_{j,n}p_n$ converges to $x_j$ for each $j$. For each $n \in \mathbb{N}$, define
\[ A_{k,n} : =  \{ g \in G_n \vert   d(gp_n,p_n) \leq 1/k  \}   ,   \]
which is clearly finite and symmetric.
\begin{center}
\textbf{Claim:} $A_{k,n}^2 \subset \bigcup _{j=1}^{C} g_{j,n} A_{k,n} $ for $n$ large enough.
\end{center}
Otherwise, after passing to a subsequence, there are $h_n \in A_{k,n}^2  \backslash  \bigcup _{j=1}^C g_{j,n} A_{k,n} $. By Lemma \ref{lem:short-powers}, after passing again to a subsequence, we can assume $h_n p_n $ converges to a point $ x_0  \in B(p, 3/k)$. By (\ref{eq:doubling-0}), there is $j_0 \in \{ 1, \ldots , C \}$ such that $ x_0 \in B^X( x_{j_0} , 1/k ) $.   Set $u_n : = g_{j_0, n}^{-1}h_n$ and compute
\[ d(u_n p_n , p_n )  = d(g_{j_0,n}p_n , h_np_n )  \to d(x_{j_0}, x_0)  < 1/k. \]
Then $u_n \in A_{k,n}$ for $n$ large enough, hence $h_n = g_{j_0, n}u_n \in g_{j_0,n} A_{k,n}$, a contradiction.

Let  $N_k : = N(C) \in \mathbb{N}$ be given by Theorem \ref{thm:margulis}, and set
\[  S_n : = \{ g \in G_n \vert d(gp_n,p_n) \leq 3 \diam (X_n/G_n) \},           \]
 which  generates $G_n$ by  Lemma \ref{lem:short-generators}. As $\diam (X_n / G_n ) \to 0$, by Lemma \ref{lem:short-powers} one has $S_n^{N_k} \subset A_{k,n}$ for $n$ large enough and Theorem \ref{thm:margulis} applies. Using also Lemma  \ref{lem:diameter-subgroup} we deduce there are subgroups $G_{k,n} \leq G_n$ with  $\diam (X_n / G_{k,n}) \to 0 $ and
 \begin{equation}\label{eq:anil0}
 G_{k,n}^{(N_k)} \subset \{ g \in G_{k,n} \vert d(gp_n,p_n) \leq 4/k \}
\end{equation}   
for $n$ large enough. By Lemma \ref{lem:normal-stabilizers}, we can upgrade (\ref{eq:anil0}) to
\begin{equation}\label{eq:anil1}
   G_{k,n}^{(N_k)} \subset \{ g \in G_{k,n} \vert d(gx,x) \leq 5/k \text{ for all }x \in X_n \}         
\end{equation}
for $n$ large enough. Replacing $G_n$ by $G_{k,n} $, we can assume the sequence $G_n$ itself satisfies (\ref{eq:anil1}). Performing the above construction for each $k \in \mathbb{N}$, one obtains the desired groups $G_n^{\prime} : = G_{k(n),n}$ via a diagonal procedure with $k (n) \to \infty$ slowly enough.
\end{proof}

\begin{remark}
\rm  From the proof of Lemma \ref{lem:almost-nilpotent}, it follows that if $C$ can be taken independent of $k$ (for example if $X$ is a Riemannian manifold), then $N$ can be taken independent of $\varepsilon $ and moreover $[G_i : G_i^{\prime } ] \leq N$ for all $i$. 
\end{remark}

\begin{proof}[Proof of Theorem \ref{thm:part-i}:]
Let $G_n^{\prime}$ be given by  Lemma \ref{lem:almost-nilpotent}. After passing to a subsequence, the groups $G_n^{\prime}$ converge equivariantly to a group $\Gamma \leq \iso (X)$ acting transitively on $X$. By Theorem \ref{thm:gleason-yamabe}, there is an open subgroup $\mathcal{O} \leq \Gamma $ with the property that for any  neighborhood $U \subset \mathcal{O}$ of $\id_X$,  there is a compact normal subgroup $K \triangleleft \mathcal{O}$ with $K \subset U$  and such that $\mathcal{O}/K$ is a connected Lie group. 
\begin{center}
\textbf{Step 1:} If $K \triangleleft \mathcal{O}$ is a compact normal subgroup such that $\mathcal{O} / K $ is a connected Lie group, then $\mathcal{O} / K$ is nilpotent.
\end{center}
Denote by $\rho : \mathcal{O} \to \mathcal{O}/K$ the projection, and  let $V \subset \mathcal{O}/K$ be a small open neighborhood of the identity such that any subgroup of $\mathcal{O}/K$ contained in $V$ is trivial. Since $\rho ^{-1}(V) \subset \mathcal{O}$ is an open neighborhood of the identity, there is $\varepsilon > 0 $ such that 
\[ \{ g \in  \mathcal{O} \vert  d(gx,x) \leq \varepsilon \text{ for all }x \in X \}      \subset \rho^{-1}(V) .  \]
By Remark \ref{rem:commutator-passes-to-limit},  there is $N \in \mathbb{N}$, such that 
\[ \mathcal{O} ^{(N)}  \subset \{  g \in \mathcal{O}  \vert d(gx,x) \leq \varepsilon \text{ for all } x \in X  \} \subset \rho^{-1}(V) ,   \]
hence $ (\mathcal{O} / K ) ^{(N)} =  \rho \left(  \mathcal{O}^{(N)} \right) \subset V $ is trivial.
\begin{center}
\textbf{Step 2:} If $K \triangleleft \mathcal{O}$ is a  compact subgroup, then $[\mathcal{O} , K] $ is trivial.
\end{center}
Let $U \subset \mathcal{O}$ be a neighborhood of $\id_X$, and $K_1 \triangleleft \mathcal{O}$ a compact normal subgroup with $K_1 \subset U$ such that $\mathcal{O}/K_1$ is a connected Lie group. By Step 1, $\mathcal{O}/K_1$ is nilpotent, and by Lemma \ref{lem:compact-nilpotent}, $KK_1/K_1 \leq \mathcal{O}/K_1$ is central, hence $[\mathcal{O}, K] \leq K_1 \subset U$. Since $U$ was arbitrary, the commutator $[\mathcal{O}, K]$ is trivial.
\begin{center}
\textbf{Step 3:} $ \mathcal{O}$ is nilpotent.
\end{center}
Take $K \triangleleft \mathcal{O}$ a compact normal subgroup such that $\mathcal{O}/K$ is a connected Lie group. By Step 1, there is $N \in \mathbb{N}$ such that $\mathcal{O}^{(N)} \leq K$. By Step 2, $ \mathcal{O}^{(N+1)} = [ \mathcal{O}, \mathcal{O}^{(N)}] \leq [\mathcal{O}, K] = \{ \id_X \}.$
\begin{center}
\textbf{Step 4:} The map $\mathcal{O} \to X$ given by $g \mapsto gp$ is a homeomorphism.
\end{center}
By Theorem \ref{thm:berestovskii-open}, $\mathcal{O}$ acts transitively on $X$, hence $X \cong \mathcal{O} / K$, where $K : = \{ g \in \mathcal{O} \vert gp = p \}$ is the stabilizer of $p$. By Step 2, $K$ is central in $\mathcal{O}$, so by Lemma \ref{lem:normal-stabilizers}, it is trivial.

Combining Steps 3 and 4 the result follows.
\end{proof}

\section{Semi-locally-simply-connected nilpotent groups}\label{sec:slsc}

In this section, we prove the second part of Theorem \ref{thm:main-theorem}, consisting of the following result.

\begin{theorem}\label{thm:part-ii}
\rm Let $(X_n,p_n)$ be a sequence of almost homogeneous spaces, converging in the pointed Gromov--Hausdorff sense to a space $(X,p)$. If $X$ is semi-locally-simply-connected, then $X$ is a Lie group equipped with an invariant sub-Finsler metric.
\end{theorem}

Due to the following result of Valerii Berestovskii \cite[Theorem 3]{BerII}, all we need to show is that $X$ is a Lie group.

\begin{theorem}
\rm  (Berestovskii) Let $X$ be a proper geodesic space whose isometry group acts transitively. If $X$ is homeomorphic to a topological manifold, then its metric is given by a sub-Finsler structure.
\end{theorem}

For the proof of Theorem \ref{thm:part-ii}, we require the following elementary observation.

\begin{lemma}\label{lem:short-loop-nt}
\rm Let $X$ be a  proper semi-locally-simply-connected geodesic space. Assume the inclusion $B(x,r) \to X$ has nontrivial content for some $x \in X$, $r > 0$. Then there is a non-contractible loop  in $X$ based at $x$ of length $\leq 3r $. 
\end{lemma}

\begin{proof}
By hypothesis, there is a loop $\beta : [0,1] \to B^X(x,r) $ based at $x$ that is non-contractible in $X$. Using the semi-local-simple-connectedness, we can find a Lipschitz loop $\gamma : [0,1] \to B^X(x,r)$ homotopic to $\beta $ by approximating it with piece-wise geodesics. Let $m \in \mathbb{N}$ be such that $\leng ( \gamma  \vert_{\left[ \frac{k-1}{m}, \frac{k}{m} \right]} ) \leq r $ for all  $k \in \{ 1, \ldots , m \}$. 

For each $k \in \{ 0, \ldots , m  \}$, let $\sigma _k : [0,1] \to B^X(x,r)$ be a minimizing path from $x$ to $\gamma (\frac{k}{m})$. Since $\gamma $ is homotopic to the concatenation of the curves $\sigma_{k-1} \ast \gamma  \vert_{\left[ \frac{k-1}{m}, \frac{k}{m} \right]} \ast \overline{\sigma}_k $ with $k \in \{ 1, \ldots , m \}$, then at least one of them is non-contractible in $X$.
\end{proof}

\begin{proof}[Proof of Theorem \ref{thm:part-ii}:]
By Theorem \ref{thm:part-i}, we can assume $X$ is a connected nilpotent group with $e_X = p$. By Corollary \ref{cor:glushkov}, if $X$ is not a Lie group, then it contains a sequence of non-trivial compact subgroups $K_1 \geq K_2 \geq \ldots  $ with 
\[   \bigcap_{j=1}^{\infty} K_j = \{ p \},    \]
such that $H_j : = X/K_j$ is a connected nilpotent Lie group, and the identity connected component of $K_j/K_{j-1}$ is non-trivial for infinitely many $j$.

Since $X$ is semi-locally-simply-connected, there is $\delta > 0$ such that the inclusion $B^X(p , \delta ) \to X $ has no content. By our assumption, there is $j \in \mathbb{N}$ with $ K_{j-1} \subset B^X(p, \delta / 3) $, and the identity connected component of $K_{j-1}/K_{j} $ is non-trivial. By Lemmas \ref{lem:compact-content} and \ref{lem:short-loop-nt}, there is a non-contractible loop $\gamma : [0,1] \to  H_{j}$ based at $e$, with $\leng (\gamma )\leq \delta $.  By Lemma \ref{lem:horizontal-lift}, there is a lift $\tilde{\gamma}_1 : [0,1] \to X$ with $\tilde{\gamma}_1 (0) = p$, $\rho \circ \tilde{\gamma}_1 = \gamma $, and $\leng (\tilde{\gamma}_1) \leq \delta $, where $\rho : X \to H_j$ is the natural projection.  For each $m \in \mathbb{N}$, define the curve $\tilde{\gamma}_m : [0,1]\to X$ as 
\[\tilde{\gamma}_m(t) : = \left[ \tilde{\gamma}_1(1) \right]^{m-1} \tilde{\gamma}_1(t). \] 
Observe that $\tilde{\gamma}_{m}(1) = \tilde{\gamma}_{m+1}(0) \in K_j$ for each $m$, so we can define the curves $\beta_m : = \tilde{\gamma}_1 \ast \ldots \ast \tilde{ \gamma }_m$, and their images lie all in $B^X(p, \delta)$.

Since $H_{j}$ is a Lie group, there is $  \varepsilon  > 0 $ such that if two closed curves in  $H_{j}$ are at uniform distance  less than $\varepsilon$, then they are homotopic to each other. Let $m_0$ be a positive integer such that $\beta_{m_0} (1) \in B^X(p, \varepsilon )$. It exists as $B^X(p,\delta ) $ is pre-compact, and $\beta_{n-m} (1) = \beta_n(1)\beta_m(1)^{-1}$ for all $n,  m  \in \mathbb{N}$ with $n \geq m$.

Let $\beta : [0,1 ] \to B^X(p,\varepsilon )$ a minimizing curve from $\beta_{m_0}(1)$ to $p$. As $\pi_1(H_j)$ has no torsion and $\rho ( \beta_{m_0} \ast \beta )$ is $\varepsilon$-uniformly close to a reparametrization of $\gamma \ast \cdots \ast \gamma$ ($m_0$ times), it is non-contractible in $H_j$. However, it factors through $B^X(p, \delta )$, meaning that the composition $B^X(p, \delta ) \to X \to H_j$ has non-trivial content, a contradiction.
\end{proof}

\section{Almost translational behavior}\label{sec:trans}

As stated in the Summary, in the remaining sections, we prove the following result, finishing the proof of Theorem \ref{thm:main-theorem}.

\begin{theorem}\label{thm:part-iii}
    \rm   Under the conditions of Theorem \ref{thm:part-ii}, for large enough $n$ there are quotients of $\pi_1(X_n)$ containing isomorphic copies of $\pi_1(X)$.
\end{theorem}

By contradiction, after passing to a subsequence, we can assume for no $n$, the group $\pi_1(X_n)$ admits a quotient containing an isomorphic copy of $\pi_1(X)$. Let $G_n^{\prime}$ be the groups given by Lemma \ref{lem:almost-nilpotent}. After passing further to a subsequence, we can assume the groups $G_n^{\prime }$ converge equivariantly to a closed group $\Gamma \leq \iso (X) $, which  by Theorems \ref{thm:part-ii} and \ref{thm:mz-hilbert} is a Lie group. The main result of this section is the following.

\begin{proposition}\label{prop:connected}
\rm $\Gamma$ acts freely on $X$.
\end{proposition}

\begin{proof} Since $\Gamma$ is a Lie group, there is a neighborhood $U \subset \Gamma $ of $\id _X $ that contains no non-trivial subgroups and there is $\varepsilon > 0 $ such that
\[ \{ g \in \Gamma \vert d(gx,x) \leq \varepsilon \text{ for all } x \in X  \} \subset U    .      \]
 By Remark \ref{rem:commutator-passes-to-limit}, there is $N \in \mathbb{N}$ such that $\Gamma ^{(N)} \subset U$, so $\Gamma $ is nilpotent of step $\leq N$. Let $\mathcal{O} \leq \Gamma $ be the identity connected component, and $K \leq \Gamma$ a compact subgroup.
 \begin{center}
 \textbf{Claim:} The commutator $[\mathcal{O}, K] $ is trivial.
 \end{center}
For the proof of this claim, we use the commutator estimates of 
 \cite[Section 3.3]{BK}. Let $\mathfrak{g}$ be the Lie algebra of $\Gamma$. By Weyl's unitary trick, we can equip $\mathfrak{g}$ with an inner product $\langle \cdot  , \cdot \rangle $ for which the adjoint action $\Adjoint : K \to \genlin (\mathfrak{g}) $ consists of orthogonal transformations. If the claim fails, there is $h \in K$ for which $\Adjoint _h : \mathfrak{g} \to \mathfrak{g} $ is not the identity. Then there is an  $\langle \cdot  , \cdot \rangle $-orthonormal basis 
\[ \{ a_1, b_1, \ldots , a_{k_1} , b_{k_1}, c_1, \ldots , c_{k_2}, d_1 , \ldots , d_{k_3} \} \hookrightarrow \mathfrak{g}  \] 
and angles 
\[ \theta_1, \ldots , \theta_{k_1} \in \mathbb{S}^1 \backslash \{ 1 \} \]
 such that $k_1 + k_2 > 0$ and
\begin{align*}
  \Adjoint_h (a_j) &= \cos \theta_j a_j + \sin \theta_j b_j,   \\
 \Adjoint_h (b_j) &=   -\sin \theta_j a_j + \cos \theta_j b_j, \\
  \Adjoint_h (c_j)& = -c_j ,\\
   \Adjoint_h (d_j ) & =   d_j. 
\end{align*} 
We deal first with the case $k_1 > 0 $. By the Baker--Cambell--Hausdorff formula, for every $\varepsilon > 0 $ there is $\delta > 0$ such that
\begin{equation}\label{eq:bch-estimate}
 d(  \log ( \exp (\delta a_1)   \exp (x) ) , x + \delta a_1 ) \leq \varepsilon \delta      
\end{equation}  
and 
\begin{equation}\label{eq:taylor}
d( \log (  h \exp (x) h^{-1} ) ,  \Adjoint _ h (x)  )  \leq \varepsilon \delta 
\end{equation}
for all $x \in B^{\mathfrak{g}}_d(0, 100 N 2^N \delta )$, where $d $ is the metric induced from $\langle \cdot , \cdot \rangle$. Iterating the estimates (\ref{eq:bch-estimate}) and (\ref{eq:taylor}), one can find  $C ( N ) >0$ such that 
\begin{equation}\label{eq:commutator-estimates}
d( \log ( \underbrace{[ h , [ h , \ldots [ h, \exp (\delta a_1)]\ldots ]]}_{\text{step }N \text{ commutator}}) , ( \id_{\mathfrak{g}} - \Adjoint _{h^{-1}}  )^N (\delta a_ 1 ) ) \leq C \varepsilon \delta .              
\end{equation}   
Since $\Gamma ^{(N)} = \{ \id _X \} $, the step $N$ commutator $[ h , [ h , \ldots [ h, \exp (\delta a_1)]\ldots ]] $ is trivial. On the other hand, a direct computation shows 
\[ \vert( \id_{\mathfrak{g}} - \Adjoint _{h^{-1}}  )^N (\delta a_ 1 ) \vert = \vert 1  - \theta_1  \vert ^N \delta , \]
 contradicting (\ref{eq:commutator-estimates}) if $\varepsilon $ is small enough (depending on $C,$ $N,$ and $\vert 1 - \theta_1  \vert $).  The case $k_1 = 0$ is similar, but using $c_1$ instead of $a_1$.

Let $K : = \{ g \in \Gamma \vert gp = p \} $ be the stabilizer of $p$ and let $x \in X$. By Theorem \ref{thm:berestovskii-open},  $\mathcal{O} $ acts transitively on $X$ so there is $g \in \mathcal{O}$ with $gp = x$. By the claim above, $gK g^{-1} = K$, hence $hx=x$ for all $h \in K$. Since $x$ was arbitrary, $K $ is trivial and $\Gamma = \Gamma / K  \cong X $.
\end{proof}

\section{Getting rid of small subgroups}\label{sec:nss}

In this section we identify and get rid of the small subgroups of $G_n^{\prime}$, using the \textit{escape norm} and the Gleason lemmas from \cite{BGT}. Let $\varphi_n : G_n^{\prime} \to \Gamma$ be the Gromov--Hausdorff approximations given by Definition \ref{def:equivariant}.  Since $\Gamma $ is a Lie group, there is $\varepsilon_0 > 0 $ such that $B^{\Gamma } (\id_X , 10^3 \varepsilon_0)$ contains no non-trivial subgroups and the inclusion $B^{\Gamma } (\id_X , 10^3 \varepsilon_0) \to \Gamma $ has no content. 

Let $B $ be a small open convex symmetric set in the Lie algebra  $\mathfrak{g} $ of $\Gamma $ such that $  \exp(B )\subset B( \id _X , \varepsilon_0) $. Notice there is $C \in \mathbb{N}$ depending only on the dimension of $\Gamma$ such that if $B$ is small enough, then one can cover $\exp (3 B)$ by $C$ translates of $\exp (B) $.
 With this $B\subset \mathfrak{g}$ and $C \in \mathbb{N}$,  define the sets 
\begin{align*}
\Theta_n & : =  B^{G_n^{\prime}} ( \id_{X_n} , 10^2 \varepsilon_0   )      \\ 
\hat{ T} _n  & : =  \{ g \in \Theta_n  \vert \varphi_n(g) \in \exp (B)   \} ,\\
 \hat{ \Sigma }_n  & :=   \{ g \in \Theta_n  \vert \varphi_n (g) \in \exp \left( B / 10^5 C^3 \right)   \} , \\
  T_n  &  := \hat{T}_n \cup \hat{T}_n^{-1} ,\\
     \Sigma _n & : =  \hat{ \Sigma }_n \cup \hat{ \Sigma }_n^{-1}. 
\end{align*}
If $B$ was chosen small enough, by the Baker--Campbell--Hausdorff formula and (\ref{eq:equivariant-composition}) one has
\[ \{ asa^{-1} \vert a \in T_n^4, s \in  \Sigma _n \} \subset  \Sigma _n^2 \]
for $n$ sufficiently large, and all three conditions of a strong global approximate group hold (see \cite[Proposition 7.3]{BGT} for further details).

\begin{lemma}\label{Strong}
\rm For $n$ sufficiently large, the set $T_n$ (thanks to the set $ \Sigma _n$) is a strong $C$-approximate group.
\end{lemma}

Lemma \ref{Strong} and Theorem \ref{thm:gleason}(\ref{thm:gleason-product}-\ref{thm:gleason-conjugate}) imply that for $n$ sufficiently large, the set
\[ W_n  := \{ g \in \Theta_n \vert   \Vert g \Vert _{T_n}=0 \} \]
is a subgroup of $G_n^{\prime}$ normalized by $T_n$.
\begin{remark}\label{rem:wn-small}
\rm From (\ref{eq:equivariant-composition}) and the fact that $B^{\Gamma } (\id _X , 10^3 \varepsilon_0 )$ has no non-trivial subgroups it is not hard to prove that for any choice of $w_n \in W_n$ one has $\varphi_n (w_n) \to \id_X$.
\end{remark}

\begin{proposition}\label{prop:gh-close}
\rm The quotient maps $X_n \to X_n/W_n$ are global $\varepsilon_n$-approximations with $\varepsilon_n \to 0$ as $n \to \infty$.
\end{proposition}
\begin{proof}
By Proposition \ref{prop:connected}, there is $\delta_0 > 0$ such that $\{ g \in \Gamma \vert d(gp,p) \leq \delta_0 \} \subset \exp (B / 2)$. Hence $ \{ g \in G_n^{\prime} \vert d(gp_n,p_n) \leq \delta_0 \} \subset T_n$ for $n$ large enough, and by Lemma \ref{lem:short-generators}, $T_n$ generates $G_n^{\prime} $ for $n$ large enough. Then $ W_n $ is a normal subgroup of $G_n^{\prime}$, so by Lemma \ref{lem:normal-stabilizers} and Remark \ref{rem:wn-small} one has
\[    \lim_{ n \to \infty }  \sup _{h \in W_n} \sup_{x \in X_n} d(hx,x)  = 0   \]
and the result follows.
\end{proof}

Set $\Gamma_n  := G_n^{\prime}/W_n $,   $\rho : G_n^{\prime} \to \Gamma_n$ the quotient map, $A_n := \rho (T_n)$, and let $\textbf{A} $ be the  ultraproduct $\lim_{n \to \alpha}  A_n$.

\begin{proposition}\label{prop:a-nss}
\rm $\textbf{A}$ is an NSS ultra approximate group.
\end{proposition}
\begin{proof}
Since $T_n$ are $C$-approximate groups, then so are the sets $A_n$.  For $[g] \in A_n \backslash \{  e \},$ we have $\Vert g  \Vert_{T_n}\neq 0$, so $g^m $ is not in $T_n^2 \supset T_n W_n $ for some $m$. Hence $   [g]^m  = \rho (g^m)$ does not belong to $A_n$.  This shows that $A_n$ does not contain non-trivial subgroups for $n$ large enough.
\end{proof} 
 We still have the map
\begin{center}
$ \overline{ \varphi }_n: \Gamma_n  \to \Gamma $
\end{center}  given by 
\begin{center}
$ \overline{ \varphi  }_n \left( [g]\right) := \varphi_n (g)$.
\end{center}
Of course, to make this map well defined, we have to choose one representative from each class in $ \Gamma_n $. However, different choices of representatives only change the value of $\overline{\varphi}_n$  by an error which goes to $0$ as $n \to \infty$. More precisely, if one considers two sequences $g_n, g_n^{\prime} \in \Theta_n $ with $g_nW_n = g_n^{\prime} W_n $, then there is a sequence $w_n \in W_n$ with $ g_n = g_n^{\prime} w_n  $ for all $n$ and by Remarks \ref{rem:equivariant-composition} and \ref{rem:wn-small} one has 
\[ \lim _{n \to \infty } d( \varphi_n (g_n) , \varphi_n (g_n^{\prime}) ) = 0.  \]
Consequently, for all $a_n, a_n^{\prime} \in A^8_n$ we have
\begin{equation}\label{eq:model-morphism}
 \lim _{n \to \infty } d( \overline{\varphi}_n (a_na_n^{\prime}) , \overline{\varphi}_n (a_n)\overline{\varphi}_n(a_n^{\prime}) ) = 0. 
\end{equation}   
 Consider the map  $\sigma : \textbf{A}^8 \to \Gamma $
given by the metric ultralimit
\[ \sigma ( \{  g_n  \}_n ) := \lim\limits_{n \to \alpha} \overline{\varphi}_n (g_n). \] 
From (\ref{eq:model-morphism}) and the fact that the maps $\varphi_n$ are Gromov--Hausdorff approximations, one has the following.
\begin{proposition}\label{prop:good-model}
\rm The pair $(\Gamma , \sigma ) $ is a good model for $\textbf{A}$.
\end{proposition}
Let $\tilde{ \Gamma } $ be the universal cover of $\Gamma $. By our choice of $\varepsilon _ 0$, the balls $B^{\tilde{\Gamma}}(e, 10^3 \varepsilon _ 0 ) $ and $B^{\Gamma}(\id_X , 10^3 \varepsilon_0 )$ are naturally identified, so for large enough $n$, we have maps $\tilde{\varphi }_n : \rho ( \Theta _n ) \to \tilde{\Gamma }$ with  $\tilde{\varphi}_n ( g ) \in B^{\tilde{\Gamma }}(e, 10^3 \varepsilon_0 )  $ and $\Phi ( \tilde{\varphi } _n (g)  ) = \overline{\varphi}_n (g)$ for all $g \in \rho ( \Theta_n )$, where $\Phi : \tilde{\Gamma } \to \Gamma$ is the natural projection.

\section{The nilprogressions}\label{sec:nilprog}

In this section, we apply a short basis procedure by Breuillard, Green, and Tao to find a large ultra nilprogression in $\textbf{A}$ (cf. \cite[Theorem 9.3]{BGT}).

\begin{theorem}\label{thm:nilprogression-construction}
\rm Let $\textbf{A}=\lim_{n \to \alpha} A_n$ be an NSS ultra approximate group. Assume there is a good Lie model $\sigma: \textbf{A}^8 \to \Gamma $. Then $\textbf{A}^4$ contains a nondegenerate ultra nilprogression $\textbf{P}$ of rank $r := \dimensi (\Gamma )$ in $C$-normal form for some $C>0$, with the property that for all standard $\varepsilon \in (0,1)$, there is an open set $U_{\varepsilon} \subset \Gamma $ with $\sigma^{-1}(U_{\varepsilon}) \subset G(\varepsilon \textbf{P})$.
\end{theorem}

\begin{proof}
The proof is done by induction on $r$. If $r=0$, then $\Gamma $ is a trivial group and the set $U_0 $ from Definition \ref{def:model} equals $ \Gamma $. Property \ref{def:model-u0} implies that $ \textbf{A}^8 = \sigma^{-1} (\Gamma )=\sigma^{-1}(U_0)  \subset \textbf{A}  $. Hence for $\alpha$-large enough $n$, $A_n$ is a group, which is trivial by the NSS property and there is nothing to show. For the induction step with $r \geq 1$, we follow step by step the construction of \cite[Section 9]{BGT}.

Let $B $ be a small open convex symmetric set in $\mathfrak{g}$, the Lie algebra of $\Gamma$.  Let  $\textbf{A}^{\prime \prime \prime} \subset \textbf{A}^{\prime \prime} \subset \textbf{A}^{\prime } \subset \textbf{A}$ be sub ultra approximate groups of $\textbf{A}$ such that 
\begin{center}
$   \sigma ^{-1} (\exp (B )) \subset \textbf{A}^{\prime} \subset \sigma^{-1} (  \exp ((1.001) B)    ),     $

$   \sigma ^{-1} (\exp (\delta B)) \subset \textbf{A}^{\prime \prime } \subset \sigma^{-1} (  \exp ((1.001) \delta B)    )    , $

$   \sigma ^{-1} (\exp ( \delta  B /10)) \subset \textbf{A}^{\prime \prime \prime } \subset \sigma^{-1} (  \exp ((1.001) \delta B/10)    ),     $
\end{center}
where $\delta \in (0,1)$ will be chosen later (their existence is guaranteed by \ref{def:model-approximation}).   Notice that if $B$ was chosen small enough, then $\textbf{A}^{\prime} $, $\textbf{A}^{\prime \prime}$, $\textbf{A}^{\prime \prime\prime} $ are strong ultra approximate groups.

Let $u \in \textbf{A}^{\prime } \backslash \{  e \}$ be such that minimizes $ \Vert u \Vert _{\textbf{A}^{\prime}}$ (in this setting, $\Vert \cdot \Vert_{\textbf{A}^{\prime}}$ is a nonstandard real number). Then, by Theorem \ref{thm:gleason}(\ref{thm:gleason-commutator}), if $\delta$ was chosen small enough,  for all $x \in (\textbf{A}^{\prime \prime})^{10}$ we have 
 \begin{center}
 $ \Vert   [u, x] \Vert _{\textbf{A}^{\prime}} = O\left( \Vert u \Vert _{\textbf{A}^{\prime}} \Vert x \Vert _{\textbf{A}^{\prime}} \right)  < \Vert u \Vert _{\textbf{A}^{\prime}}  $.
 \end{center}
Since $\Vert u \Vert_{\textbf{A}^{\prime}}$ was minimal, $u $  commutes with every element in $(\textbf{A}^{\prime \prime })^{10}$. Consequently, if we define 
\[ \textbf{Z} := \{   u^k   \vert \, k \in {}^{\ast}\mathbb{N}, \,  \vert k \vert \leq 1/ \Vert u \Vert _{\textbf{A}^{\prime}}  \} ,\]
then every element of $\textbf{Z}$ will commute with every element of $(\textbf{A}^{\prime \prime })^{10}$. Since $\left( \textbf{A}^{\prime \prime} \right)^6$ is well defined, by Lemma \ref{lem:local-quotient} we can form the quotients $\textbf{A}^{\prime \prime} /\textbf{Z} $ and $\textbf{A}^{\prime \prime \prime } /\textbf{Z} $. 

\begin{lemma}
\rm The image $\sigma (\textbf{Z} )$ is of the form $\phi ( [-1,1] )$, with $\phi (t) := \exp (tv)$,  for some non-zero $v$ in  the center of $\mathfrak{g}$. If $\delta$ is small enough,  after taking a small open neighborhood of the identity $U\subset \Gamma $ and the quotient $U/\sigma (\textbf{Z}) $, one can guarantee that
\begin{enumerate}[label=\roman*]
\item $U/\sigma (\textbf{Z})$ is a connected local Lie group of dimension $r-1$.
\item $\textbf{A}^{\prime\prime} / \textbf{Z}$ and $\textbf{A}^{\prime\prime\prime}/\textbf{Z}$ are NSS ultra-approximate groups.
\item $\overline{\sigma} : \left(\textbf{A}^{\prime\prime }/\textbf{Z} \right)^8 \to U / \sigma (\textbf{Z})$ is a good model, where $\overline{\sigma }$ comes from the composition $(\textbf{A}^{\prime\prime})^8 \to U \to U / \sigma (\textbf{Z})$.
\end{enumerate}
\end{lemma}

\begin{proof}
This is the content of \cite[Lemma 9.4 \small{(}\normalsize{i-ii}\small{)}]{BGT} and its proof.
\end{proof}
 We can then apply the induction hypothesis to $\textbf{A}^{\prime \prime \prime}/\textbf{Z}$ and deduce there is a non-degenerate ultra nilprogression $\overline{\textbf{P}}: =  \overline{P}(\overline{u}_1, \ldots , \overline{ u}_{r-1} ; \overline{N_1}, \ldots , $ $\overline{N}_{r-1}  )$  in $\overline{C}$-normal form with $\overline{C}> 0 $,  $u_j \in \textbf{A}^{\prime\prime } / \textbf{Z}$, $\overline{N}_j \in {}^{\ast}\mathbb{R}$ for $j \in \{ 1, \ldots , r-1 \}$ such that $\overline{\textbf{P}} \subset (\textbf{A}^{\prime \prime \prime}/ \textbf{Z})^4 \subset \textbf{A}^{\prime \prime}/\textbf{Z}$ and having the property that for all standard $\varepsilon > 0$, there is an open set $V_{\varepsilon} \subset U / \sigma (\textbf{Z})$ with $( \overline{\sigma})^{-1} (V_{\varepsilon}) \subset G(\varepsilon \overline{\textbf{P}})$.
 
 For $\theta > 0$, construct $\textbf{P} : =P( u_1, \ldots , u_{r};$ $ N_1, \ldots , N_{r})$, where $u_j \in \textbf{A}^{\prime \prime}$ is a lift of $\overline{u}_j$ that minimizes $\Vert u_j \Vert _{\textbf{A}^{\prime \prime }}$ for $j \in \{ 1, \ldots , r-1\}$,   $N_j : = \theta \overline{N}_j$ for $j \in \{1, \ldots , r-1\}$, $u_{r} : = u$, $N_{r }: = \theta / \Vert u \Vert _{\textbf{A}^{\prime \prime}}$. By \cite[Lemma 9.4 \small{(}\normalsize{iii}\small{)}]{BGT} and its proof, if $\theta $ is small enough, $\textbf{P}$ is a non-degenerate ultra nilprogression in $C$-normal form for some $C > 0$. 
 
  The only thing left to prove is that for all $\varepsilon > 0$, there is an open set $U_{\varepsilon} \subset \Gamma $ such that $\sigma ^{-1}  (U_{\varepsilon}) \subset G( \varepsilon \textbf{P})$. By contradiction, assume that for some $\varepsilon > 0$, there is  $x \in \textbf{A}^{\prime \prime }\backslash G( \varepsilon \textbf{P})$ with $\sigma (x) = e_{\Gamma }$. If that is the case, $ \overline{\sigma } (x \textbf{Z}) = e_{U/ \sigma (\textbf{Z})}$, and by our induction hypothesis, for all standard $\eta >0$ we have $\pi (x) \in G( \eta \overline{\textbf{P}})$, where $\pi : \textbf{A}^{\prime \prime } \to \textbf{A}^{\prime \prime } / \textbf{Z} $ is the natural projection. Therefore $x = u_1 ^{n_1} \ldots u_{r} ^{n_r}$, with
\begin{center}
 $\vert n_j \vert \leq \eta \overline{N_j} /\overline{C}$ for $j \in \{ 1, \ldots , r-1 \} $, $\vert n_r \vert \leq \Vert u_r \Vert _{\textbf{A}^{\prime}}$.
\end{center}
By \cite[Lemma 9.5]{BGT}, one has $\Vert u_j \Vert _{\textbf{A}^{\prime \prime }} = O (  \Vert \overline{u}_j \Vert _{\textbf{A}^{\prime \prime }/\textbf{Z}} )$ for $j \in \{ 1, \ldots , r-1 \}$. Combining this with the fact that $\overline{N}_j = O(1/\Vert   \overline{u}_j \Vert _{\textbf{A}^{\prime \prime }/\textbf{Z}})$ for each $j \in \{ 1, \ldots , r-1 \}$ and Theorem \ref{thm:gleason}(\ref{thm:gleason-product}), we get
\begin{align*}
\Vert u_1^{n_1} \ldots u_{r-1}^{n_{r-1}} \Vert _{\textbf{A}^{\prime \prime}} &=  O \left(  \sum_{j=1}^{r-1} \Vert u_j ^{n_j} \Vert_{\textbf{A}^{\prime \prime}}  \right)\\
& =   O\left( \sum_{j=1}^{r-1} \vert n_j \vert \Vert u_j \Vert_{\textbf{A}^{\prime \prime}} \right) \\
& =  O \left( \eta \sum_{j=1}^{r-1}  \overline{N_j} \Vert \overline{u}_j \Vert_{\textbf{A}^{\prime \prime}/\textbf{Z}} \right) \\
& =  O (\eta).
\end{align*}
Since $\eta$ was arbitrary, we obtain that $\Vert u_1^{n_1} \ldots u_{r-1}^{n_{r-1}} \Vert _{\textbf{A}^{\prime \prime}}$ is infinitesimal. Then again by Theorem \ref{thm:gleason}(\ref{thm:gleason-product}),
\begin{center}
$\Vert u_{r} ^{n_{r}} \Vert _{ \textbf{A}^{\prime \prime} } = O(  \Vert x \Vert_{\textbf{A}^{\prime \prime}} + \Vert u_1^{n_1} \ldots u_{r-1}^{n_{r-1}} \Vert _{\textbf{A}^{\prime \prime}}   )   $.
\end{center}
This implies that $   \Vert u_{r } ^{n_{r}} \Vert _{ \textbf{A}^{\prime \prime} }   $ is infinitesimal so $\vert n_{r} \vert = o (N_{r} ) \leq \varepsilon N_{r} / C $. Also, since $\eta$ was arbitrary, $\vert n_j \vert \leq \varepsilon N_j /C$ for $j\in \{1, \ldots , r-1 \}$. Therefore $x \in G(\varepsilon \textbf{P})$, which is a contradiction.
\end{proof}

\begin{remark}\label{strong-basis}
\rm From the proof of Theorem \ref{thm:nilprogression-construction}, the group $\Gamma $ is nilpotent and the basis $\{  v_1, \ldots , v_r \}   $  of $ \mathfrak{ g}$ given by
\begin{equation}\label{eq:malcev-basis}
\exp ( t v_j ) = \sigma \left( u_j ^{\left\lfloor  t  N_j /C \right\rfloor  }   \right)  \text{ for }t \in [0,1]  
\end{equation}
is a strong Malcev basis (see also \cite[Proposition 9.6]{BGT}). 
\end{remark}
\section{Malcev theory}\label{sec:malcev}

By Propositions \ref{prop:a-nss} and \ref{prop:good-model}, Theorem \ref{thm:nilprogression-construction} applies to $\textbf{A}$. Let $r : = \dim (\Gamma)$, $\textbf{P} = P (u_1, \ldots , u_r ;$ $ N_1, \ldots , N_r)$, and $C>0$ be given by Theorem \ref{thm:nilprogression-construction}, and let $\{ v_1, \ldots , v_r \} \subset \mathfrak{g}$ be the Malcev basis given by (\ref{eq:malcev-basis}). With these parameters, let $\varepsilon > 0 $ be given by Theorem \ref{thm:malcev-embedding}. By \cite[Remark C.4]{BGT}, there is $\delta > 0 $ such that $G( \delta \textbf{P} )^2 \subset G( \varepsilon \textbf{P} )$.    We fix these choices of $C $,  $\varepsilon$, and $\delta$ for the rest of this paper. Let $\tilde{\Gamma }$ be the universal cover of $\Gamma$. By Theorem \ref{thm:simply-connected-nilpotent}, the map $\psi : \mathbb{R}^r \to \tilde{\Gamma }$ given by 
 \begin{equation}\label{eq:psi-def}
   \psi (x_1, \ldots, x_r   ) : =  \exp (\delta x_1 v_1) \ldots \exp (\delta x_r v_r)
\end{equation} 
is a diffeomorphism.

\begin{lemma}\label{lem:yiren}
\rm  The group structure $Q: \mathbb{R}^r \times \mathbb{R}^r \to \mathbb{R}^r$ given by 
\[ Q(x,y) := \psi^{-1} (  \psi (x) \psi (y) )  \]
is a quasilinear polynomial of degree $\leq d(r)$. We will denote the group $(\mathbb{R}^r, Q)$ as $H$.
\end{lemma}

\begin{proof}
By the Baker--Campbell--Hausdorff formula, after identifying $\mathfrak{g}$ with $\mathbb{R}^r$ via the basis $\{ v_1, \ldots , v_r \}$, the map $\mathbb{R}^r \times \mathbb{R}^r \to \mathbb{R}^r$ given by
 \[ (x,y) \to \log (  \psi (x) \psi (y)  )  \]
is polynomial of degree $\leq r$. Also, by Theorem \ref{thm:simply-connected-nilpotent}, the map $\mathbb{R}^r \to \mathbb{R}^r$ given by 
\[ x \to \psi ^{-1} ( \exp (x)  ) \]
is polynomial of degree bounded by a number depending only on $r$. Therefore the composition is also polynomial of degree $\leq d(r)$. Quasilinearity is immediate from the definition.
\end{proof}

By  Theorem \ref{thm:malcev-embedding}, for $\alpha$-large enough $n$, the nilprogressions $P_n$ are good with Malcev polynomials $\hat{Q}_n$ and define the groups $\tilde{\Gamma}_n : = \Gamma_{P_n} = (\mathbb{Z}^r , \hat{Q}_n\vert_{\mathbb{Z}^r \times \mathbb{Z}^r} )$.  Let $N_0\in \mathbb{N}$ be given by Lemma \ref{lem:well-finite} with $d (r)$ given by Lemma \ref{lem:yiren} and Theorem \ref{thm:malcev-embedding}, and define  $\xi : \mathbb{N}\to \mathbb{N}$ as
\begin{equation}\label{eq:xi-def}
\xi (n) : = N_0 \left\lfloor \frac{\delta n}{CN_0} \right\rfloor.  
\end{equation}   
For $n \in \mathbb{N}$, consider $  \kappa_n : \mathbb{R}^r \to \mathbb{R}^r  $ given by
\[ \kappa_n (x_1, \ldots , x_r) : = (x_1 \xi (N_{1,n}), \ldots ,  x_r \xi (N_{r,n} )) ,   \]
where $N_j = \{ N_{j,n } \}_n $ for $j \in \{ 1, \ldots , r \}$.  Let $H_n$ be the group $(\mathbb{R}^r , Q_n)$, where $Q_n : \mathbb{R}^r \times \mathbb{R}^r \to \mathbb{R}^r  $ is the group structure given by
\[  Q_n (x,y) : = \kappa_n ^{-1} (\hat{Q}_n ( \kappa_n(x) ,\kappa_n(y)) ) . \]

\begin{proposition}\label{prop:q-good}
\rm  The sequence of quasilinear polynomial group structures $Q_n$ converges well to $Q$. 
\end{proposition}

Define $\Omega  \subset \mathbb{R}^r$ as
\[ \Omega : = \left\{  -1, \ldots,\frac{-1}{N_0}, 0, \frac{1}{N_0}, \ldots , 1    \right\} ^{\times r} . \]
 Consider the maps $\omega_{\alpha} :  \Omega  \to \tilde{\Gamma }$ and $\omega_n :  \Omega  \to \tilde{\Gamma}_n$ defined as $\omega_{\alpha} : = \psi \vert _{ \Omega }$ and
 \begin{equation}\label{eq:omega-n-def}
   \omega_n ( x_1, \ldots, x_r) :=  \gamma_1^{ x_1 \xi (N_{1,n})  }\ldots   \gamma_r^{ x_r \xi ( N_{r,n})  } .
\end{equation} 
 We also define maps $\varphi^{\flat}_n : G(\varepsilon P_n )^{\sharp } \to \Gamma $ and $ \tilde{\varphi} ^{\flat }_n  : G(\varepsilon P_n)^{\sharp} \to \tilde{\Gamma }$ as 
\begin{equation}\label{eq:flat-def}
 \varphi^{\flat}_n  (  x^{\sharp } ) := \overline{\varphi }_n (x)  \text{ and } \tilde{\varphi}^{\flat}_n  (  x^{\sharp } ) := \tilde{\varphi } _n (x) . 
\end{equation}  
Consider the following diagram.

\begin{displaymath}
  \begin{tikzcd}[column sep=3em]
   \Omega \times \Omega \arrow{r}{ \omega_n } \arrow{d}{ \id} & (G (\delta P_n )^{\sharp})^{\times 2} \arrow{r}{\ast} \arrow{d}{ \tilde{\varphi}^{\flat}_n}&  G (\varepsilon P_n )^{\sharp} \arrow{d}{ \tilde{\varphi}^{\flat}_n } \arrow{r}{\kappa_n^{-1}}  & \mathbb{R}^r  \arrow{d}{\id} \\
  \Omega \times \Omega \arrow{r}{\omega_{\alpha} } & \tilde{\Gamma }\times \tilde{\Gamma }  \arrow{r}{\ast }& \tilde{\Gamma } \arrow{r}{ \psi^{-1} } &  \mathbb{R}^r  
  \end{tikzcd}
\end{displaymath}
The first row of the diagram is the polynomial $Q_n$, while the second row is the polynomial $Q$. Commutativity of the diagram does not hold in general, but it holds in the limit, as the following proposition (and its proof) shows. 
\begin{lemma}\label{lem:com-d}
\rm  For every $x,y \in \Omega$,
\[  \lim\limits_{n \to \alpha} \kappa_n^{-1} \left(   \omega_n(x) \omega_n(y)   \right)    =   \psi^{-1} (  \omega_{\alpha}(x)\omega_{\alpha}(y))  .\]
\end{lemma}

\begin{proof}
 If $x = (x_1, \ldots , x_r) \in \Omega$, one has 
\begin{equation}\label{eq:omegas}
\begin{split}
\omega_{\alpha }(x) & =  \exp (\delta x_1 v_1) \ldots \exp (\delta x_r v_r)\\
& =   \lim_{n \to \alpha}  \tilde{\varphi}^{\flat}_n \left( \gamma_1^{ x_1 \xi ( N_{1,n})   }  \right) \ldots \lim_{n \to \alpha}\tilde{\varphi}^{\flat}_n \left( \gamma_r^{ x_r \xi (  N_{r,n}) }  \right)\\
& =  \lim_{n \to \alpha}  \tilde{\varphi}^{\flat}_n (\omega_n(x)),
\end{split}
\end{equation}
where we used (\ref{eq:psi-def}) on the first line, (\ref{eq:malcev-basis}) and (\ref{eq:xi-def}) on the second one, and (\ref{eq:model-morphism}) and (\ref{eq:omega-n-def}) on the third one.  On the other hand, for any sequence $x_n^{\sharp} \in G(\varepsilon P_n)^{\sharp} \subset \tilde{\Gamma}$, we can decompose it as 
\[   x_n = x_{1,n}\cdots x_{r,n} ,  \]
with $x_{j,n} = u_{j,n}^{p_{j,n}} , \, \vert p_{j,n } \vert \leq \varepsilon N_{j,n} / C $ for each $j \in \{ 1, \ldots , r \}$. Then we have
\begin{equation}\label{eq:psi-kappa}
\begin{split}
 \lim\limits_{n \to \alpha} \tilde{\varphi}^{\flat}_n (x_n^{\sharp} ) & =  \lim\limits_{n \to \alpha} \tilde{\varphi}_n  (x_{1,n}) \ldots \lim\limits_{n \to \alpha} \tilde{\varphi}_n (x_{r,n})\\
 & =  \exp \left(  \lim\limits_{n \to \alpha} \dfrac{Cp_{1,n}}{ N_{1,n}} v_1  \right)\ldots \exp \left(  \lim\limits_{n \to \alpha} \dfrac{C p_{r,n}}{ N_{r,n}} v_r  \right) \\
 & =  \psi \left( \dfrac{C}{\delta} \lim\limits_{n \to \alpha} \left( \dfrac{p_{1,n}}{N_{1,n}}, \ldots , \dfrac{p_{r,n}}{N_{r,n}}   \right) \right) \\
 & =  \psi \left( \lim\limits_{n \to \alpha} \kappa_n^{-1} (  x_n ^{\sharp} ) \right) ,
\end{split}
\end{equation}
where we used (\ref{eq:model-morphism}) in the first line, (\ref{eq:malcev-basis}) on the second one, (\ref{eq:psi-def}) on the third one, and (\ref{eq:xi-def}) on the fourth one. Then we conclude
\begin{align*}
\omega_{\alpha}(x) \omega_{\alpha}(y) & =  \lim_{n \to \alpha}  \tilde{\varphi}^{\flat}_n (  \omega_n (x)  ) \lim_{n \to \alpha}  \tilde{\varphi}^{\flat}_n (  \omega_n (y)  ) \\
& =  \lim_{n \to \alpha}  \tilde{\varphi}^{\flat}_n (\omega_n (x) \omega_n (y))\\
& =  \psi \left(  \lim\limits_{n \to \alpha} \kappa_n^{-1} \left(   \omega_n(x) \omega_n(y)   \right)    \right),
\end{align*}
where we used (\ref{eq:omegas}) on the first line, (\ref{eq:model-morphism}) on the second one, and (\ref{eq:psi-kappa}) on the third one.
\end{proof}
\begin{proof}[Proof of Proposition \ref{prop:q-good}:]
Apply Lemmas \ref{lem:well-finite} and \ref{lem:com-d}.
\end{proof} 
 By Proposition \ref{prop:connected} and Theorem \ref{thm:nilprogression-construction}, there is $\rho > 0 $ such that 
\begin{equation}\label{eq:sigma-n}
 S_n : = \{ g \in \Gamma_n \vert d( g [p_n], [p_n] ) < \rho \} \subset G( \delta P_n )    
\end{equation}  
for $\alpha$-large enough $n$. Let $\tilde{\Gamma }_n^{\prime }$ be the abstract group generated by $S_n$, with relations
\begin{center}
$ s= s_1s_2 \in \tilde{\Gamma }_n^{\prime } $ whenever  $ s,s_1, s_2 \in S_n$ and $s=s_1s_2$  in $\Gamma_n$. 
\end{center}
From Theorem \ref{thm:malcev-embedding}, for $\alpha$-large enough $n$ we have $\tilde{\Gamma }_n^{\prime } = \tilde{\Gamma }_n$, and by Theorem \ref{thm:monodromy} there is a regular  $(\rho /4)$-wide covering map $\tilde{X}_n \to X_n/W_n$ whose Galois group is the kernel of the canonical map $\Phi_n : \tilde{ \Gamma}_n \to \Gamma _ n$.  From Proposition \ref{prop:gh-close} and Theorem \ref{thm:cover-copy} we also get the following result.
\begin{proposition}\label{prop:phi-is-quotient}
\rm For $\alpha$-large enough $n$, $\kker (\Phi_n)$ is a quotient of $\pi_1(X_n)$.
\end{proposition}
\begin{remark}\label{rem:phi-locally-injective}
\rm  From (\ref{eq:sigma-n}), if $n$ is $\alpha$-large enough, then for every $g \in \tilde{\Gamma}_n $ with  $d ( \Phi_n (g) [p_n] , [p_n]  )   < \rho ,$  there is a unique $w \in G( \delta P_n )^{\sharp}$ with $gw \in \kker (\Phi_n)$. In particular,  $\kker(\Phi_n ) \cap S_n ^{\sharp } = \{ e_{ \tilde{\Gamma}_n } \}$.
\end{remark}

\section{Almost torsion elements}\label{sec:dictionary}

In this last section we finish the proof of Theorem \ref{thm:part-iii} (and consequently Theorem \ref{thm:main-theorem}) with the following result.
\begin{proposition}\label{prop:phi-contains-pi1x}
\rm For $\alpha$-large enough $n$, $\kker ( \Phi_n )$ contains an isomorphic copy of $\pi_1(X)$. 
\end{proposition}

\begin{proof}
  Let $\Phi :  \tilde{ \Gamma } \to \Gamma $ denote the canonical projection. By Corollary \ref{cor:aspherical}, the group $\kker (\Phi ) = \pi_1( \Gamma ) \cong \pi_1(X)$  is finitely generated torsion free abelian. Let $\{ \lambda_1 , \ldots , \lambda_{\ell}   \} $ be a basis of $\kker (\Phi )$ as a free abelian group. Pick $M \in \mathbb{N}$ large enough so  that the $M$-th roots of the $\lambda_j$'s lie in the ball $B^{\tilde{\Gamma }}(e, \rho )$, with $\rho$ given by (\ref{eq:sigma-n}). For each $j \in \{ 1, \ldots , \ell \}$ pick a sequence
\[ \lambda _{j,n} \in G(\delta P_n)^{\sharp} \subset \tilde{\Gamma}_n \]
 with 
 \begin{equation}\label{eq:lambda-def}
  \lim\limits_{n \to \alpha} \tilde{\varphi}_n^{\flat} (\lambda_{j,n}) = \lambda_j ^{1/M}.
\end{equation}  
For each $j \in \{ 1, \ldots , \ell \}$ we have
 \begin{align*}
 \lim\limits_{n \to \alpha} \overline{\varphi}_n  \left( \Phi_n \left(  \lambda_{j,n} ^M  \right)  \right)  & =  \left( \lim\limits_{n \to \alpha} \overline{\varphi}_n \left( \Phi_n \left(  \lambda_{j,n}  \right)   \right)\right)^M\\
 & =   \left( \lim\limits_{n \to \alpha} \varphi_n^{\flat}  (   \lambda_{j,n} ) \right)^M   \\
 & =  \Phi \left(  \lambda_j ^{1/M} \right)^M\\
 & =  e_{\Gamma },  
 \end{align*}
where we used (\ref{eq:model-morphism}) in the first line, (\ref{eq:flat-def}) in the second one, and (\ref{eq:lambda-def}) in the third one.  By Remark \ref{rem:phi-locally-injective}, for $\alpha$-large enough $n$, and all $j \in \{ 1 , \ldots , \ell \}$, there are $w_{j,n}\in G( \delta P_n)^{\sharp}$  with  
 \[  \lim_{n \to \alpha} \tilde{\varphi}_n^{\flat}(w_{j,n}) = e_{\tilde{\Gamma}} \text{ and } \lambda_{j,n}^M w_{j,n} \in \kker (\Phi_n).\]
Consider the morphisms $ \natural : \tilde{\Gamma}_n \to H_n $ and $\natural : \tilde{\Gamma } \to H$ given by  $g^{\natural } : = \kappa_n^{-1}(g)$ for $g \in \tilde{\Gamma}_n$, and $g^{\natural}: = \psi^{-1}(g) $ for $g \in \tilde{\Gamma }$. These identifications allow us to have all these group structures in the same underlying set $\mathbb{R}^r$.  Since $Q_n$ converges well to $Q$, we deduce from (\ref{eq:lambda-def}) that
 \[   \lim\limits_{n \to \alpha} \left[ \lambda_{j,n} ^M w_{j,n}\right]^{\natural} =  \lambda_j^{\natural}  . \]
 Then by Lemma \ref{lem:log-continuous}, 
 \[  \lim\limits_{n \to \alpha} \log_n \left(  \left[ \lambda_{j,n} ^M w_{j,n}\right]^{\natural}   \right) = \log (\lambda_j^{\natural}) \text{ for each }j \in \{ 1, \ldots , \ell   \}  , \]
where $\log_n, \log,$ denote the logarithm maps with respect to $H_n$ and $H$, respectively. Therefore for $\alpha$-large enough $n$, the set
\[  \left\{ \log_n \left(  \left[ \lambda_{1,n} ^M w_{1,n}\right]^{\natural}   \right) , \ldots , \log_n \left(  \left[ \lambda_{\ell,n} ^M w_{\ell ,n}\right]^{\natural}   \right) \right\}    \]
is linearly independent. Also, for $j,k \in \{ 1, \ldots , \ell \}$, 
\begin{equation}\label{eq:commutator-almost-trivial}
\begin{split}
 \lim\limits_{n \to \alpha}  [ \lambda_{j,n}^M w_{j,i},\lambda_{k,n} ^M w_{k,i}   ]^{\natural}   &= \left[  \lim\limits_{n \to \alpha}\lambda_{j,n}^Mw_{j,n}, \lim\limits_{n \to \alpha}\lambda_{k,n}^Mw_{k,n}   \right]^{\natural}  \\
 & =  \left( \lambda_j^{-1} \lambda_k ^{-1}\lambda_j \lambda_k \right)^{\natural} \\
 & =  e_H.
 \end{split}
\end{equation}
Since $[ \lambda_{j,n}^M w_{j,i},\lambda_{k,n} ^M w_{k,i}   ] \in \kker (\Phi_n)  $, (\ref{eq:commutator-almost-trivial}) and Remark \ref{rem:phi-locally-injective} imply that  
\[  [ \lambda_{j,n}^M w_{j,n},\lambda_{k,n} ^M w_{k,n}   ]  =e_{\tilde{\Gamma}_n} \]
for $\alpha$-large enough $n$. Then by Lemma \ref{lem:li-com-free} the group 
\[   \langle   \lambda_{1,n} ^M w_{1,n}, \ldots ,  \lambda_{\ell,n} ^M w_{\ell  , n} \rangle \leq \kker \left( \Phi_n \right)      \]
is isomorphic to $\pi_1(X)$.
\end{proof}
\begin{proof}[Proof of Theorem \ref{thm:part-iii}:]
As stated at the beginning of Section \ref{sec:trans} we can assume, working by contradiction and  after passing to a subsequence, that for no $n$ the group $\pi_1(X_n)$ admits a quotient containing an isomorphic copy of $\pi_1(X)$. Combining this assumption with Propositions \ref{prop:phi-is-quotient} and \ref{prop:phi-contains-pi1x} yields a contraction.
\end{proof}



\begin{thebibliography}{ASM}


\bibitem{BFT} I. Benjamini, H. Finucane and R. Tessera, \textit{On the scaling limit of finite vertex transitive graphs with large diameter.} Combinatorica \textbf{37} (2017), no.3, 333-374.

\bibitem{BerBus} V.N. Berestovskii, \textit{Homogeneous Buseman G-spaces.} Siberian Mathematical Journal \textbf{23} (1982), no.2, 141-150.



\bibitem{BerII} V.N. Berestovskii, \textit{Homogeneous manifolds with intrinsic metric. II} Siberian Mathematical Journal \textbf{30} (1989), no.2, 180-191.


\bibitem{BGT} E. Breuillard, B. Green and T. Tao, \textit{The structure of approximate groups.}  Publ. Math. Inst. Hautes \'Etudes Sci. \textbf{116} (2012), 115-221.


\bibitem{BK} P. Buser and H. Karcher, \textit{Gromov's almost flat manifolds.} Ast\'erisque \textbf{81} Soci\'et\'e Math\'ematique de France, Paris, 1981, 148 pp.


\bibitem{Cor} L. Corwin and F. Greenleaf, \textit{Representations of nilpotent Lie groups and their applications. Part 1, Basic theory and examples.} Cambridge Studies in Advanced Math \textbf{18} Cambridge University Press, Cambridge, NY, 1990, 269 pp.


\bibitem{FY92} K. Fukaya and T. Yamaguchi, \textit{The fundamental groups of almost non-negatively curved manifolds.}  Ann. of Math. (2) \textbf{136} (1992), no.2, 253-333.

\bibitem{Gel} T. Gelander, \textit{Limits of finite homogeneous metric spaces.} Enseign. Math. (2) \textbf{59} (2013), no.1-2, 195-206.

\bibitem{Gleason} A.M. Gleason, \textit{The structure of locally compact groups.} Duke Math. J. \textbf{18} (1951) 85-104.

\bibitem{Glu} V.M. Glu\v{s}hkov, \textit{Locally bicompact groups with minimality condition for closed subgroups.} Ukrain. Mat. \v{Z}. \textbf{8} (1956), 135-139.

\bibitem{GromovPG} M. Gromov, \textit{Groups of polynomial growth and expanding maps.} Math. Inst. Hautes \'Etudes Sci. \textbf{53} (1981), 53-78.

\bibitem{GrAF} M. Gromov,  \textit{Almost flat manifolds.} J. Differential Geometry \textbf{13} (1978), no.2, 231-241.

\bibitem{GromovMS} M. Gromov, \textit{Metric structures for Riemannian and non-Riemannian spaces.} Based on the 1981 French original. With appendices by M. Katz, P. Pansu and S. Semmes. Translated from the French by Sean Michael Bates. Progr. Math., 152, Birkh\"auser Boston, Inc., Boston, MA, 1999. xx+585 pp.

\bibitem{Hr} E. Hrushovski, \textit{Stable group theory and approximate subgroups.} J. Amer. Math. Soc. \textbf{25} (2012), no.1, 189-243.


\bibitem{Luck} W. L\"uck,  \textit{Aspherical manifolds.} Bulletin of the Manifold Atlas (2012), 1-17.


\bibitem{Malcev} A.I. Malcev,  \textit{On a class of homogeneous spaces,} Izvestiya Akad. Nauk. SSSR. Ser. Mat. \textbf{13} (1949), 9-32.

\bibitem{Munk} J. Munkres, \textit{Elements of algebraic topology.} Addison-Wesley Publishing Company, Menlo Park, CA, 1984. ix+454 pp.


\bibitem{MZ} D. Montgomery and L.  Zippin, \textit{Topological transformation groups.} Interscience Publishers, New York-London, 1955, xi+282 pp.


\bibitem{Pansu} P. Pansu,  \textit{Croissance des boules et des g\'eod\'esiques ferm\'ees dans les nilvari\'et\'es.} Ergodic Theory Dynam. Systems \textbf{3} (1983), no.3, 415-445.

\bibitem{SWHau} C. Sormani and G.  Wei, \textit{Hausdorff convergence and universal covers.}  Trans. Amer. Math. Soc. \textbf{353} (2001), no.9, 3585-3602.

\bibitem{SWUni} C. Sormani and G.  Wei, \textit{Universal covers for Hausdorff limits of noncompact spaces.} Trans. Amer. Math. Soc. \textbf{356} (2004), no.3, 1233-1270.

\bibitem{SWCov} C. Sormani and G.  Wei, \textit{The covering spectrum of a compact length space.} J. Differential Geom. \textbf{67} (2004), no.1, 35-77.

\bibitem{TT} R. Tessera and M.C. Tointon, \textit{A finitary structure theorem for vertex-transitive graphs of polynomial growth.} Combinatorica \textbf{41} (2021), no.2, 263-298.
 
 
\bibitem{Tur} Turing, A. M.  \textit{Finite approximations to lie groups.} Ann. of Math. (2) \textbf{39} (1938), no.1, 105-111.

\bibitem{Yamabe} H. Yamabe,  \textit{A generalization of a theorem of Gleason.}  Ann. of Math. (2) \textbf{58} (1953), 351-365.





\end{thebibliography}
\end{document}